\documentclass[12pt]{article}
\usepackage{amsmath}
\usepackage{amsthm}
\usepackage{amssymb}
\usepackage{booktabs}
\usepackage{mathtools}
\usepackage{lmodern}
\usepackage{microtype}
\usepackage[T1]{fontenc}
\usepackage[utf8]{inputenc}
\usepackage{hyphenat}
\usepackage{enumitem}
\usepackage{mathabx}
\usepackage{xcolor}
\usepackage{comment}
\usepackage{thm-restate}
\usepackage[pdftex]{graphicx}
\usepackage[pdftex,margin=1in,marginparwidth=0.75in]{geometry}
\PassOptionsToPackage{hyphens}{url}
\usepackage[bookmarks=true, bookmarksopen=true,%
    bookmarksdepth=3,bookmarksopenlevel=2,%
    colorlinks=true,%
    linkcolor=blue,%
    citecolor=blue,%
    filecolor=blue,%
    menucolor=blue,%
    urlcolor=blue]{hyperref}
\hypersetup{pdftitle={Title}}
\hypersetup{pdfauthor={authors}}
\usepackage{url}
\usepackage{tikz}
\usepackage{cite}
\usetikzlibrary{decorations.markings, arrows, decorations.fractals}

\graphicspath{{draws/}}

\counterwithin{figure}{section}

\newtheorem{theorem}{Theorem}[section]
\newtheorem{lemma}[theorem]{Lemma}
\newtheorem{proposition}[theorem]{Proposition}
\newtheorem{conjecture}[theorem]{Conjecture}
\theoremstyle{definition}
\newtheorem{definition}[theorem]{Definition}

\begin{document}

\title{A Full Study of the Dynamics on One-Holed Dilation Tori}
\author{Mason Haberle and Jane Wang}
\date{Summer 2020}

\date{\today}

\maketitle

\begin{abstract}
    An open question in the study of dilation surfaces is to determine the typical dynamical behavior of the directional flow on a fixed dilation surface. 
    We show that on any one-holed dilation torus, in all but a measure zero Cantor set of directions, the directional flow has an attracting periodic orbit, is minimal, or is completely periodic. 
    We further show that for directions in this Cantor set, the directional flow is attracted to either a saddle connection or a lamination on the surface that is locally the product of a measure zero Cantor set and an interval.
\end{abstract}


\section{Introduction}
\label{sec:intro}

\textit{Translation surfaces} are geometric surfaces that are locally modeled by the plane with structure group the group of translations.  
One way to think about translation surfaces is to view them as collection of polygons in the plane with pairs of parallel opposite sides identified by translation.  
An example of a translation surface is shown below in Figure \ref{fig:translation}.

\begin{figure}[ht]
\begin{center}
    \begin{tikzpicture}
        \coordinate (A) at (0,0);
        \coordinate (B) at (1,0);
        \coordinate (C) at (2,1);
        \coordinate (D) at (0,2);
        \coordinate (E) at (0,3);
        \coordinate (F) at (-1,2);
        \coordinate (G) at (-2,2);
        \coordinate (H) at (-2,1);

        \fill[fill = blue!5] (A) -- (B) -- (C) -- (D) -- (E) -- (F) -- (G) -- (H) -- (A);
        \draw[blue, very thick] 
            (A) -- (B) node[pos = 0.5, color = black, below]{$A$} 
            (F) -- (G) node[pos = 0.5, color = black, above]{$A$};
        \draw[red, very thick]  
            (B) -- (C) node[pos = 0.5, color = black, right]{$B$}
            (E) -- (F) node[pos = 0.5, color = black, left]{$B$};
        \draw[yellow, very thick] 
            (C) -- (D) node[pos = 0.5, color = black, above]{$C$} 
            (A) -- (H) node[pos = 0.5, color = black, below]{$C$};
        \draw[black!30!green, very thick]
            (G) -- (H) node[pos = 0.5, color = black, left]{$D$} 
            (E) -- (D) node[pos = 0.5, color = black, right]{$D$};
        
    \end{tikzpicture}
\caption{A genus $2$ translation surface.}
\label{fig:translation}
\end{center}
\end{figure}
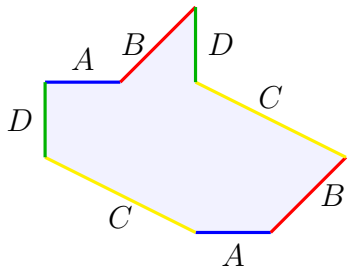

Translation surfaces are well-studied objects with connections to many other areas of mathematics including dynamical systems, algebraic geometry, and mathematical billiards. 
From the dynamical perspective, there has been much progress made on understanding the \emph{directional flow} or the \emph{directional foliations} on given and generic translation surfaces. 
For example, a celebrated theorem of Kerchoff, Masur, and Smillie \cite{kms_ergodicity} proves that on any translation surface, the directional flow in almost any direction is uniquely ergodic.

\textit{Dilation surfaces} are a natural generalization of translation surfaces. 
They are newer objects and while much less is known about them, many of the same questions that can be asked about translation surfaces can be also be asked about dilation surfaces. 
Like translation surfaces, we can think of dilation surfaces as collections of polygons in the plane, now with sides identified not just by translation but by translation and dilation.  
We call such a representation a \emph{polygonal model} for the surface.

Directional flows and foliations still exist on dilation surfaces, but one result of this generalization is that the spectrum of possible dynamical behaviors of the straight line flow on a dilation surface is richer and more difficult to understand than on a translation surface.  

For example, it is possible to find straight trajectories on dilation surfaces that converge to a periodic orbit, as seen in Figure \ref{fig:dilation}. 
This type of dynamical behavior cannot occur on translation surfaces. 

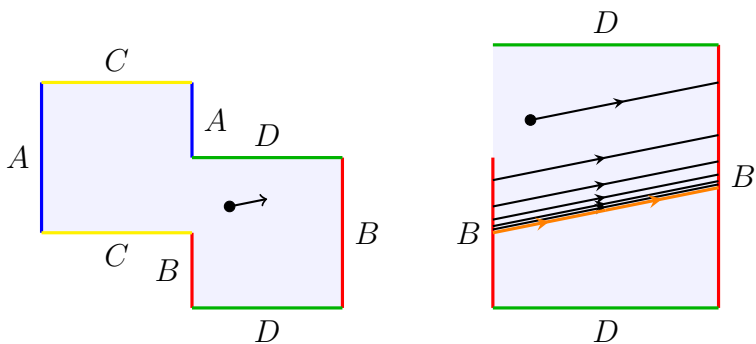
\begin{figure}[ht]
\begin{center}
    \begin{tikzpicture}[decoration = {
            markings,
            mark = at position 0.5 with {\arrow{stealth}},
        }]
    
        \coordinate (A) at (0, 1);
        \coordinate (B) at (2, 1);
        \coordinate (C) at (2, 0);
        \coordinate (D) at (4, 0);
        \coordinate (E) at (4, 2);
        \coordinate (F) at (2, 2);
        \coordinate (G) at (2, 3);
        \coordinate (H) at (0, 3);
        
        \fill[fill = blue!5] (A) -- (B) -- (C) -- (D) -- (E) -- (F) -- (G) -- (H) -- (A);
        \draw[blue, very thick] 
            (H) -- (A) node[pos = 0.5, color = black, left]{$A$} 
            (F) -- (G) node[pos = 0.5, color = black, right]{$A$};
        \draw[red, very thick]  
            (B) -- (C) node[pos = 0.5, color = black, left]{$B$}
            (D) -- (E) node[pos = 0.5, color = black, right]{$B$};
        \draw[yellow, very thick] 
            (A) -- (B) node[pos = 0.5, color = black, below]{$C$} 
            (G) -- (H) node[pos = 0.5, color = black, above]{$C$};
        \draw[black!30!green, very thick]
            (C) -- (D) node[pos = 0.5, color = black, below]{$D$} 
            (E) -- (F) node[pos = 0.5, color = black, above]{$D$};

        \coordinate (I) at (6, 0);
        \coordinate (J) at (6, 2);
        \coordinate (K) at (6, 3.5);
        \coordinate (L) at (9, 3.5);
        \coordinate (M) at (9, 0);
        
        \fill[fill = blue!5] (I) -- (K) -- (L) -- (M);
        
        \draw[red, very thick] (I) -- (J) node[pos = 0.5, left, color = black]{$B$}
        (L) -- (M) node[pos = 0.5, right, color = black]{$B$};
        \draw[black!30!green, very thick] (K) -- (L) node[pos = 0.5, above, color = black]{$D$}
        (M) -- (I) node[pos = 0.5, below, color = black]{$D$};

        \draw [thick, postaction = decorate] (6.5, 2.5) -- (9, 3);
        \draw [thick, postaction = decorate] (6, 1.7) -- (9, 2.3);
        \draw [thick, postaction = decorate] (6, 1.35) -- (9, 1.95);
        \draw [thick, postaction = decorate] (6, 1.175) -- (9, 1.775);
        \draw [thick, postaction = decorate] (6, 1.0875) -- (9, 1.6875);
        \draw [thick, postaction = decorate] (6, 1.04375) -- (9, 1.64375);
        \draw [very thick, orange, postaction = decorate] (6, 1) -- (7.5, 1.3);
        \draw [very thick, orange, postaction = decorate] (7.5, 1.3) -- (9, 1.6);
        
        \filldraw (6.5, 2.5) circle (2pt);
        \filldraw (2.5, 1.35) circle (2pt);
        
        \draw[->, thick] (2.5, 1.35) -- (3, 1.45);
        
    \end{tikzpicture}
\caption{Left: A genus $2$ dilation surface.  Right: A closeup of the surface contains a black trajectory which converges to the orange periodic orbit.}
\label{fig:dilation}
\end{center}
\end{figure}

One goal is then to understand the typical dynamical behavior of the straight line flow on a dilation surface, in the spirit of a statement like the theorem of Kerchoff-Masur-Smillie on translation surfaces. 
In \cite{dTori_teich_dynamics}, Ghazouani conjectures that on any dilation surface, the directional flow in almost every direction has an attracting and repelling periodic orbit. 

Some evidence to support this conjecture has been found by various authors by examining the dynamics on certain families of dilation surfaces.  Such examples have been found to match the conjectured behavior. 
In \cite{cascades}, Boulanger, Fougeron, and Ghazouani study a surface that they call the Disco Surface by leveraging a connection between dynamics on dilation surfaces and the dynamics of a family of maps on the interval called \textit{affine interval exchange maps} or AIETs for short (for more on AIETs, please see Subsection \ref{sec:aiets}). 
Ghazouani uses an understanding of the broader dynamics on the moduli space of dilation surfaces to investigate the dynamics on twice-punctured dilation tori in \cite{dTori_teich_dynamics}. 
In separate works by Boulanger-Ghazouani \cite{ohdTori} and Bowman-Sanderson \cite{staircase}, the authors use different techniques to investigate dynamics on one-holed dilation tori. 

Our work builds on the work in these last two papers and furthers the understanding of dynamics on one-holed dilation tori. 
Boulanger and Ghazouani claim in \cite{ohdTori} that on any one-holed dilation torus, there is a full-measure set of directions on which the directional flow accumulates on a set of closed curves and a measure-zero set of directions on which the accumulation set locally looks like a Cantor set cross an interval. 
They further claim that these results about the dynamics on general dilation tori follow verbatim from analysis of dynamics on the Disco Surface done in \cite{cascades}. 
However, we found that the methods used to analyze the Disco Surface did not directly apply to general one-holed dilation tori and that the study of one-holed dilation tori dynamics required more careful analysis. 
In our first theorem, we expand on the techniques of Boulanger and Ghazouani to carefully prove their claims about one-holed dilation tori dynamics.

\begin{restatable*}{theorem}{cantor}
\label{thm:1}
 Let $X$ be a dilation torus with one nonempty boundary component. Then there is a Cantor set of measure zero $C \subset \mathbb{RP}^1$, such that the straight line flow on $X$ in any direction in $\mathbb{RP}^1 \setminus C$ beginning at any noncritical point accumulates to a closed orbit in $X$. 
\end{restatable*}

We see from the above theorem that for almost every direction on any one-holed dilation torus, the directional foliation is simple to understand: it accumulates on a closed orbit or saddle connection. There is also a single direction in which the flow might be completely periodic or minimal. 

What is perhaps more interesting is the dynamics in the remaining Cantor set of directions $C \subseteq \mathbb{RP}^1$. 
We will see that the set $C$ of directions splits into three sets with different behaviors. 

A \emph{lamination} of a dilation surface is a nowhere-dense closed union of parallel trajectories. On a portion of the set $C$ of directions, the straight line flow will accumulate to a lamination whose cross-section is a Cantor set. This dynamical behavior of Cantor-accumulation was also claimed to exist by Boulanger and Ghazouani in \cite{ohdTori}. 
Our contribution is then to prove this carefully and to show using new techniques that the Cantor set cross-section of such a lamination is always measure zero.

\begin{restatable*}{theorem}{dynamics}
\label{thm:2}
The Cantor set $C$ of directions defined in Theorem \ref{thm:1} for a one-holed dilation torus $X$ splits into three sets $\{m_B\} \cup C_1 \cup C_2$ with the following dynamical behaviors: 
\begin{itemize}
    \itemsep0ex
    \item The direction $m_B$ is the direction parallel to the boundary of $X$. In this direction, the straight line flow can be minimal, can accumulate to a periodic trajectory, or can be completely periodic. 
    \item The set of directions $C_1$ is a countable set. In directions in $C_1$, the straight line flow on $X$ accumulates onto a saddle connection. 
    \item In directions in $C_2$, the straight line flow on $X$ accumulates to a lamination of $X$ whose cross-section is a measure zero Cantor set.
\end{itemize}
\end{restatable*}

The proof of the statement that these Cantor sets are measure zero relies on new techniques involving the measures of certain sets under finitely many iterates of an associated AIET. 
Through the analytic properties of how these measures change as we vary the parameters of the underlying AIETs, we can better understand the measures of these limiting Cantor sets.

\section{Background}
\label{sec:background}

In this section, we provide some background on dilation surfaces and affine interval exchange maps. 
Much of the motivation for studying dilation surfaces comes from the wealth of knowledge about their cousins, translation surfaces. 
There are many wonderful introductions to translation surfaces. 
We refer an interested reader to \cite{wright_intro} or \cite{flatsurfaces}.

\subsection{Dilation surfaces}
\label{sec:dilationBackground}

In line with \cite{cascades}, we define dilation surfaces in the following way:

\begin{definition}[Dilation Surface]
\label{def:dilationsurface}
A \textit{dilation surface} is a surface $X$ with a finite set of cone points $\Sigma \subset X$ and an atlas of charts on $X \setminus \Sigma$ whose transition maps are restrictions of maps $z \mapsto a z + b$ for $a \in \mathbb{R}_+$, $b \in \mathbb{C}$.  
We require that punctured neighborhoods of $X$ around points of $\Sigma$ be homeomorphic to some $k$-sheet covering of $\mathbb{C}$.
\end{definition}

We can always find a \emph{polygonal model} for the dilation surfaces of interest: 
a collection of polygons with sides identified in parallel opposite pairs by translations and/or dilations by real factors.  
For our purposes, we allow dilation surfaces to have boundary components which have no identifications.  
Some examples are found in Figures \ref{fig:doublechamber} and \ref{fig:dilationhexagon}.

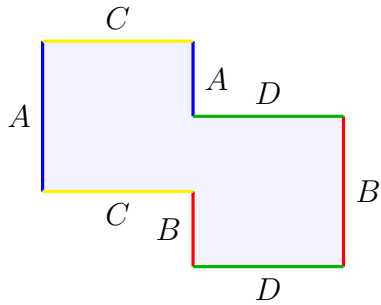
\begin{figure}[ht]
\begin{center}
    \begin{tikzpicture}
        \coordinate (A) at (0, 1);
        \coordinate (B) at (2, 1);
        \coordinate (C) at (2, 0);
        \coordinate (D) at (4, 0);
        \coordinate (E) at (4, 2);
        \coordinate (F) at (2, 2);
        \coordinate (G) at (2, 3);
        \coordinate (H) at (0, 3);
        
        \fill[fill = blue!5] (A) -- (B) -- (C) -- (D) -- (E) -- (F) -- (G) -- (H) -- (A);
        \draw[blue, very thick] 
            (H) -- (A) node[pos = 0.5, color = black, left]{$A$} 
            (F) -- (G) node[pos = 0.5, color = black, right]{$A$};
        \draw[red, very thick]  
            (B) -- (C) node[pos = 0.5, color = black, left]{$B$}
            (D) -- (E) node[pos = 0.5, color = black, right]{$B$};
        \draw[yellow, very thick] 
            (A) -- (B) node[pos = 0.5, color = black, below]{$C$} 
            (G) -- (H) node[pos = 0.5, color = black, above]{$C$};
        \draw[black!30!green, very thick]
            (C) -- (D) node[pos = 0.5, color = black, below]{$D$} 
            (E) -- (F) node[pos = 0.5, color = black, above]{$D$};
        
    \end{tikzpicture}
\caption{The double chamber, a genus 2 dilation surface with one cone point of angle $6\pi$.}
\label{fig:doublechamber}
\end{center}
\end{figure}

\begin{figure}[ht]
\begin{center}
\begin{tikzpicture}
    \coordinate (A) at (0, 0);
    \coordinate (B) at (2, 1);
    \coordinate (C) at (5, 0);
    \coordinate (D) at (5, 1.833);
    \coordinate (E) at (3, 2.5);
    \coordinate (F) at (0, 1);
    
    \fill[fill = blue!5] (A) -- (B) -- (C) -- (D) -- (E) -- (F) -- (A);
    
    \draw [very thick, color = blue] 
        (E) -- (F) node[pos = 0.5, above left, color = black]{$A$}
        (A) -- (B) node[pos = 0.5, below right, color = black]{$A$};
    \draw [very thick, color = red] 
        (B) -- (C) node[pos = 0.5, below left, color = black]{$B$}
        (D) -- (E) node[pos = 0.5, above right, color = black]{$B$};
    \draw [very thick, color = black] 
        (F) -- (A) node[pos = 0.5, left, color = black]{$C$}
        (C) -- (D) node[pos = 0.5, right, color = black]{$C$};
\end{tikzpicture}
\caption{A genus 1 dilation surface with two cone points of angle $2\pi$.}
\label{fig:dilationhexagon}
\end{center}
\end{figure}
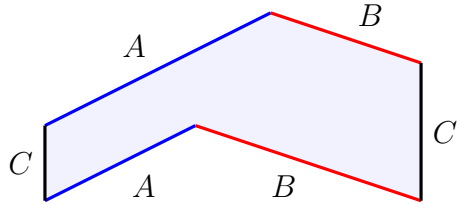

We can then consider the set of directions $\mathbb{RP}^1$ on the surface.  
For each direction $m \in \mathbb{RP}^1$, the foliation on $X$ by lines of slope $m$ is called the \emph{directional foliation} of $X$ in direction $m$.  
The leaves of the directional foliation are the orbits of points under the straight-line or geodesic flow in direction $m$.  
Leaves which intersect a cone point terminate at that point.  We call these leaves \emph{critical}.  
A leaf that intersects cone points on both ends is called a \emph{saddle connection}.

We are interested in the long-term dynamics of the geodesic flow on $X$.  
For a given directional foliation, we want to describe the $\omega$-limit (forward limit) and $\alpha$-limit (backward limit) sets of each noncritical leaf.  
For general dilation surfaces, this is a difficult task.  
In \cite{dTori_teich_dynamics}, it is conjectured that generically the foliations are \textit{Morse-Smale}, i.e. attracted to closed leaves:

\begin{definition}[Morse-Smale]
\label{def:Morse-Smale}
A foliation is \textit{Morse-Smale} if it has a finite collection of closed leaves such that the $\alpha$-limit set of any leaf of the foliation is one of these leaves, and the $\omega$-limit set of any leaf of the foliation is one of these leaves.
\end{definition}

\begin{conjecture}[S. Ghazouani]
\label{conj:5}
Let $X$ be a dilation surface which is not a translation surface.  
The collection of directions in $\mathbb{RP}^1$ with Morse-Smale directional foliations is full measure.
\end{conjecture}

Note the assumption that $X$ not be a translation surface.  
In the case that $X$ is a translation surface, i.e. all transition maps of the atlas take the form $z \mapsto z + b$, then the generic behavior is determined in \cite{kms_ergodicity}:

\begin{theorem}[S. Kerchoff, H. Masur, J. Smillie]
\label{thm:6}
Let $X$ be a translation surface (without boundary).  
There is a full measure set of directions in $\mathbb{RP}^1$ whose directional foliations have leaves which are uniquely ergodic.
\end{theorem}

In particular, such leaves are minimal and equidistributed.  
This is in stark contrast to the generic Morse-Smale behavior conjectured for dilation surfaces above.  
While a full proof of the conjectured generic dynamical behavior still seems far-off, it is possible to ascertain generic Morse-Smale behavior for simple classes of dilation surfaces.  
To do so, we pass to the machinery of affine interval exchange transformations.

\subsection{Affine Interval Exchange Transformations}
\label{sec:aiets}

To study dynamics of the straight line flow on a dilation surface, it sometimes helps to consider instead the dynamics of the first return map on a transversal to the flow. 
As we will see, this first return map will be an example of an \textit{affine interval exchange transformation}. 
In this way, we can pass from studying dynamics of a flow on a surface to studying the dynamics of a map on the interval. 

This technique of passing to a lower-dimensional dynamical system has also been used extensively in the study of translation surfaces. 
In the case of translation surfaces, the first return map to a transversal is an interval exchange transformation. 
A good introduction to this subject is \cite{yoccoz}. 

\begin{definition}[Affine Interval Exchange Map] 
\label{def:AIET}
An affine interval exchange map (AIET) is an injective map $T$ from an interval to itself constructed by breaking up the interval into finitely many pieces and then mapping these pieces back to non-overlapping subintervals of the interval by affine maps $x \mapsto ax+b$ for $a \geq 0$. 
\end{definition}

Because there is some ambiguity as to where the endpoints of each interval map are mapped, we generally ignore the endpoints when thinking about AIETs. 

It is informative to provide some examples of AIETs.  
We will use a graphical representation with the domain pictured on the top and the range on the bottom.  
$T$ carries each subinterval on top to the corresponding subinterval on bottom by an affine, orientation-preserving mapping.  
In Figure \ref{fig:AIETex1} we see an example of a surjective AIET where the original interval is split up into three pieces that are then permuted and scaled. 

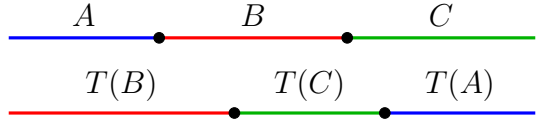
\begin{figure}[ht]
\begin{center}
\begin{tikzpicture}
  \coordinate (A) at (0, 1);
  \coordinate (B) at (2, 1);
  \coordinate (C) at (4.5, 1);
  \coordinate (D) at (7, 1);
  \coordinate (E) at (0, 0);
  \coordinate (F) at (3, 0);
  \coordinate (G) at (5, 0);
  \coordinate (H) at (7, 0);

  \draw [very thick, color = blue] (A) -- (B) node[pos = 0.5, above, color = black]{$A$};
  \draw [very thick, color = red] (B) -- (C) node[pos = 0.5, above, color = black]{$B$};
  \draw [very thick, color = black!30!green] (C) -- (D) node[pos = 0.5, above, color = black]{$C$};
  \draw [very thick, color = blue] (G) -- (H) node[pos = 0.5, above, color = black]{$T(A)$};
  \draw [very thick, color = red] (E) -- (F) node[pos = 0.5, above, color = black]{$T(B)$};
  \draw [very thick, color = black!30!green] (F) -- (G) node[pos = 0.5, above, color = black]{$T(C)$};
  
  \filldraw [color = black]
    (B) circle (2pt)
    (F) circle (2pt)
    (C) circle (2pt)
    (G) circle (2pt);
\end{tikzpicture}
\caption{A surjective AIET on 3 intervals.}
\label{fig:AIETex1}
\end{center}
\end{figure}

We can also consider AIETs that are not surjective. 
Figure \ref{fig:AIETex2} shows an AIET on three intervals where the image intervals do not cover the whole interval again. 

\begin{figure}[ht]
\begin{center}
\begin{tikzpicture}
  \coordinate (A) at (0, 1);
  \coordinate (B) at (2, 1);
  \coordinate (C) at (4.5, 1);
  \coordinate (D) at (7, 1);
  \coordinate (E) at (0, 0);
  \coordinate (F) at (3, 0);
  \coordinate (G) at (5, 0);
  \coordinate (H) at (7, 0);

  \draw [very thick, color = blue] (A) -- (B) node[pos = 0.5, above, color = black]{$A$};
  \draw [very thick, color = red] (B) -- (C) node[pos = 0.5, above, color = black]{$B$};
  \draw [very thick, color = black!30!green] (C) -- (D) node[pos = 0.5, above, color = black]{$C$};
  \draw [very thick, color = blue] (G) -- (H) node[pos = 0.5, above, color = black]{$T(A)$};
  \draw [very thick, color = red] (E) -- (1,0) node[pos = 0.5, above, color = black]{$T(B)$};
  \draw [very thick, color = black!30!green] (F) -- (G) node[pos = 0.5, above, color = black]{$T(C)$};
  
  \filldraw [color = black]
    (B) circle (2pt)
    (C) circle (2pt)
    (G) circle (2pt);
\end{tikzpicture}
\caption{A non-surjective AIET on 3 intervals.}
\label{fig:AIETex2}
\end{center}
\end{figure}
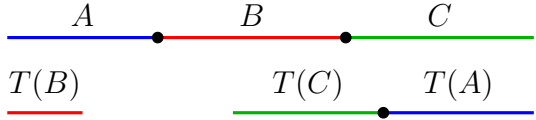

To understand dynamics on a dilation surface in a particular direction, it is often sufficient to understand the dynamics of the first return map on a transversal to the direction of flow. 
In the case when the first return map exists for every point on the transversal, the first return map will always be an AIET. 
This correspondence is best seen through an example.

Suppose that we wanted to understand the dynamics of the vertical flow on the dilation surface in Figure \ref{fig:frmexample}. 
We might instead consider the first return map on the horizontal transversal shown. 
That is, for every point $x$ on the transversal, we let $T(x)$ be the point where the vertical flow from $x$ on the surface first returns back to the transversal. 

\begin{figure}[ht]
\begin{center}
\begin{tikzpicture}
    \coordinate (A) at (0, 1);
    \coordinate (B) at (3, 0);
    \coordinate (C) at (7, 1);
    \coordinate (D) at (7, 2.5);
    \coordinate (E) at (4.5, 3.333);
    \coordinate (F) at (0, 2.208);
    
    \fill[fill = blue!5] (A) -- (B) -- (C) -- (D) -- (E) -- (F) -- (A);
    
    \draw [very thick, color = blue] 
        (E) -- (F) node[pos = 0.5, above left, color = black]{$A$}
        (B) -- (C) node[pos = 0.5, below right, color = black]{$A$};
    \draw [very thick, color = red] 
        (A) -- (B) node[pos = 0.5, below left, color = black]{$B$}
        (D) -- (E) node[pos = 0.5, above right, color = black]{$B$};
    \draw [very thick, color = black!30!green] 
        (F) -- (A) node[pos = 0.5, left, color = black]{$C$}
        (C) -- (D) node[pos = 0.5, right, color = black]{$C$};
    \draw [thick] (0, 1.5) -- (7, 1.5);
        
    \filldraw
        (2.25, 1.5) circle (2pt) node[below]{$x$}
        (5, 1.5) circle (2pt) node[above]{$T(x)$};
        
    \draw[very thick, ->] (2.25, 1.5) -- (2.25, 2.25);
    \draw[very thick] (2.25, 2.25) -- (2.25, 2.76);
    \draw[very thick, ->] (5, 0.5) -- (5, 1);
    \draw[very thick] (5, 1) -- (5, 1.5);
\end{tikzpicture}
\vskip 0.5cm
\begin{tikzpicture}
    \coordinate (A) at (0, 1);
    \coordinate (B) at (4.5, 1);
    \coordinate (C) at (7, 1);
    \coordinate (D) at (0, 0);
    \coordinate (E) at (3, 0);
    \coordinate (F) at (7, 0);
    
    \draw [very thick, color = blue] (A) -- (B) node[pos = 0.5, above, color = black]{$A$};
    \draw [very thick, color = red] (B) -- (C) node[pos = 0.5, above, color = black]{$B$};
    \draw [very thick, color = blue] (E) -- (F) node[pos = 0.5, above, color = black]{$T(A)$};
    \draw [very thick, color = red] (D) -- (E) node[pos = 0.5, above, color = black]{$T(B)$};
    
    \filldraw [color = black] (B) circle (2pt);
    \filldraw [color = black] (E) circle (2pt);
\end{tikzpicture}
\caption{A dilation surface with a vertical flow, and the first return map on a transversal.}
\label{fig:frmexample}
\end{center}
\end{figure}
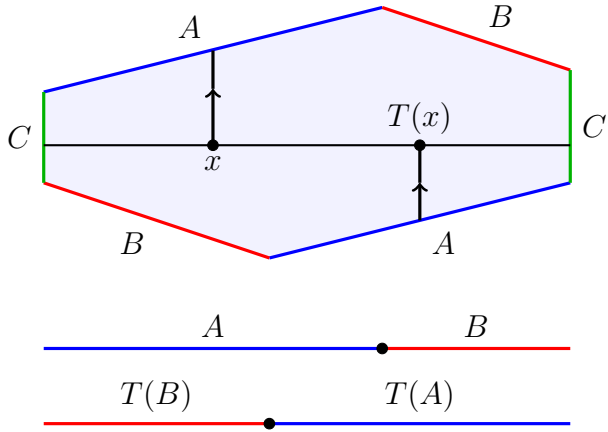

If we repeat this process for every point on the transversal, we find that our first return map $T$ is the AIET on two intervals in Figure \ref{fig:frmexample}.

The dynamics of this $2$-AIET directly correspond to the dynamics of the vertical flow on the dilation surface since the vertical flow is a suspension of the $2$-AIET.  
If the $2$-AIET is minimal, has an attracting periodic orbit or critical orbit (orbit that eventually hits the break point $x_T$), is completely periodic, or has some other behavior, then the vertical flow on the dilation surface must have the same behavior.
In this way, we have reduced our problem of understanding dynamics on a dilation surface to understanding the dynamics of a corresponding AIET.

\section{Flows on a one-holed dilation torus}
\label{sec:flowStructure}

We aim to study the behavior of straight line flows on a particular class of dilation surfaces known as \emph{one-holed dilation tori}.  
By a one-holed dilation torus, we mean a dilation surface with one boundary component, one cone point on the boundary of angle $3\pi$, and no cone points on the interior of the surface.  
The boundary should be a straight segment connecting the cone point to itself.  
Topologically, such a surface is a torus with an open disk removed.  
A typical polygonal model for a one-holed dilation torus is shown in Figure \ref{fig:dilationtorus}.

\begin{figure}[ht]
\begin{center}
\begin{tikzpicture}
  \coordinate (A) at (1, 0);
  \coordinate (B) at (0, 1.5);
  \coordinate (C) at (1, 3.5);
  \coordinate (D) at (5, 3.5);
  \coordinate (E) at (3.25, 0);

  \fill[fill = blue!5] (A) -- (B) -- (C) -- (D) -- (E) -- (A);
  
  \draw [very thick] (A) -- (B);
  \draw [very thick, color = blue] (B) -- (C) node[pos = 0.5, above left, color = black]{$A$};
  \draw [very thick, color = red] (C) -- (D) node[pos = 0.5, above, color = black]{$B$};
  \draw [very thick, color = blue] (D) -- (E) node[pos = 0.5, below right, color = black]{$A$};
  \draw [very thick, color = red] (E) -- (A) node[pos = 0.5, below, color = black]{$B$};
  
  \filldraw
  (A) circle (2pt)
  (B) circle (2pt)
  (C) circle (2pt)
  (D) circle (2pt)
  (E) circle (2pt);
\end{tikzpicture}
\caption{A one-holed dilation torus.  The black side is the boundary component.}
\label{fig:dilationtorus}
\end{center}
\end{figure}
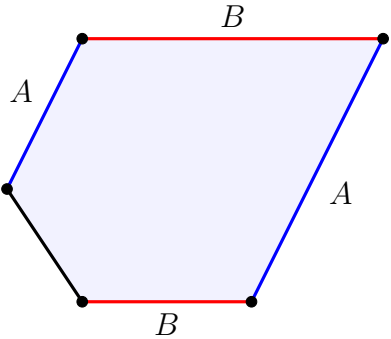

In fact, any one-holed dilation torus has a polygonal model given by a convex pentagon, after a triangulation and a cut-and-paste operation.
See Figure \ref{fig:concavetorus} for an example.

\begin{figure}[ht]
\begin{center}
\begin{tikzpicture}
  \coordinate (A) at (1, 0);
  \coordinate (B) at (-1, -1);
  \coordinate (C) at (1, 1.5);
  \coordinate (D) at (4.45, 1.5);
  \coordinate (E) at (3.25, 0);
  
  \fill [fill = blue!5] (A) -- (B) -- (C) -- (D) -- (E) -- (A);

  \draw [very thick] (A) -- (B);
  \draw [very thick, color = blue] (B) -- (C) node[pos = 0.5, above left, color = black]{$A$};
  \draw [very thick, color = red] (C) -- (D) node[pos = 0.5, above, color = black]{$B$};
  \draw [very thick, color = blue] (D) -- (E) node[pos = 0.5, below right, color = black]{$A$};
  \draw [very thick, color = red] (E) -- (A) node[pos = 0.5, below, color = black]{$B$};
  \draw [thick, dashed] (A) -- (C);
  
  \filldraw
  (A) circle (2pt)
  (B) circle (2pt)
  (C) circle (2pt)
  (D) circle (2pt)
  (E) circle (2pt);

  \coordinate (A) at (7, 0);
  \coordinate (B) at (10.45, 0.8);
  \coordinate (C) at (7, 1.5);
  \coordinate (D) at (10.45, 1.5);
  \coordinate (E) at (9.25, 0);
  
  \fill [fill = blue!5] (A) -- (C) -- (D) -- (B) -- (E) -- (A);

  \draw [very thick, color = blue] (A) -- (C) node[pos = 0.5, left, color = black]{$A$};
  \draw [very thick, color = red] (C) -- (D) node[pos = 0.5, above, color = black]{$B$};
  \draw [very thick, color = blue] (D) -- (B) node[pos = 0.5, right, color = black]{$A$};
  \draw [very thick] (B) -- (E);
  \draw [very thick, color = red] (E) -- (A) node[pos = 0.5, below, color = black]{$B$};
  \draw [thick, dashed] (D) -- (E);
  
  \filldraw
  (A) circle (2pt)
  (B) circle (2pt)
  (C) circle (2pt)
  (D) circle (2pt)
  (E) circle (2pt);
\end{tikzpicture}
\caption{The first polygonal model is concave, but after cutting the left triangle and pasting on the right we arrive at a convex polygonal model.}
\label{fig:concavetorus}
\end{center}
\end{figure}
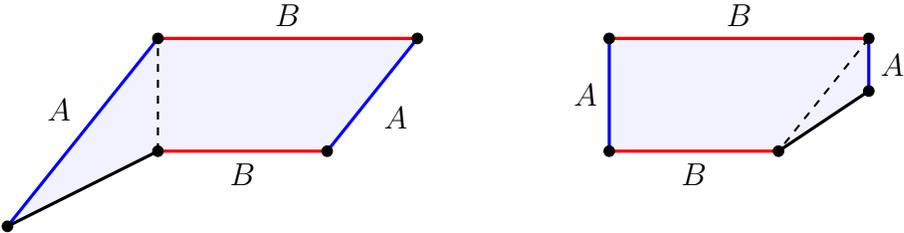

$X$ will denote a one-holed dilation torus.  To make sure that trajectories do not collide with the boundary, we only analyze the semicircle of directions pointing away from the boundary.  We call these the forward trajectories and parameterize them by their slope $m \in \mathbb{RP}^1$.  We can understand the dynamics of the forward trajectories by studying their first return maps on selected transverse diagonals of $X$. 

$\mathbb{RP}^1$ may be broken into six intervals, each containing directions whose first return maps obey one of four superficial behaviors.  The endpoints of these intervals are the directions of certain saddle connections on the surface.  In Figure \ref{fig:directionIntervals}, each interval is shown with endpoints colored according to their parallel saddle connection.  The first ``interval'' contains only a single direction, that which is parallel to the boundary.  We call this direction $m_B$.  The remaining intervals are labeled $I_1, \dots, I_5$ in clockwise order (behaviors at endpoints will be described by the analysis on either side).

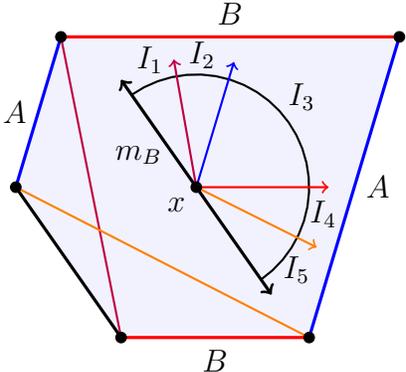
\begin{figure}[ht]
\begin{center}
\begin{tikzpicture}
    \pgftransformxslant{0.3}
    \pgftransformscale{2}
    \coordinate (A) at (1, 0);
    \coordinate (B) at (2.25, 0);
    \coordinate (C) at (2.25, 2);
    \coordinate (D) at (0, 2);
    \coordinate (E) at (0, 1);
    
    \fill[fill = blue!5] (A) -- (B) -- (C) -- (D) -- (E) -- (A);
    
    \draw [very thick, color = blue] 
        (B) -- (C) node[pos = 0.5, right, color = black]{$A$}
        (D) -- (E) node[pos = 0.5, left, color = black]{$A$};
    \draw [very thick, color = red] 
        (A) -- (B) node[pos = 0.5, below, color = black]{$B$}
        (C) -- (D) node[pos = 0.5, above, color = black]{$B$};
    \draw [very thick] (E) -- (A);
    \draw [thick, purple] (A) -- (D);
    \draw [thick, orange] (E) -- (B);
    
    \pgftransformxslant{-0.3}
    \coordinate (X) at (1.5, 1);
    \draw [thick,domain=-55:125] plot ({1.5 + 0.75*cos(\x)}, {1 + 0.75*sin(\x)});
    \draw [very thick, ->] (X) -- (1, 1.714) node[black, above right, pos = 0.95]{$I_1$} node[black, left, pos = 0.3]{$m_B$};
    \draw [very thick, ->] (X) -- (2, 0.286) node[black, above right, pos = 1]{$I_5$};
    \draw [purple, thick, ->] (X) -- (1.35, 1.85) node[black, above right, pos = 0.84]{$I_2$};
    \draw [blue, thick, ->] (X) -- (1.75, 1.833);
    \draw [red, thick, ->] (X) -- (2.38, 1);
    \draw [orange, thick, ->] (X) -- (2.3, 0.6) node[black, above right, pos = 0.85]{$I_4$};
    \filldraw (X) circle (1pt) node[below left]{$x$} ;
    \node[black] at (2.2, 1.6) {$I_3$};

    \filldraw (A) circle (1pt)
        (B) circle (1pt)
        (C) circle (1pt)
        (D) circle (1pt)
        (E) circle (1pt);
        
\end{tikzpicture}
\caption{The intervals of directions in a dilation torus.}
\label{fig:directionIntervals}
\end{center}
\end{figure}

\subsection{The boundary-parallel direction}
\label{sec:boundaryParallelDirection}

First, we determine the behaviors of the flow in the boundary-parallel direction $m_B$.
This flow is special in that it produces a bijective first return map on a transverse diagonal.  
We note that it may also be reversed without risk of hitting the boundary.  
The picture of the first return map $T$ of the flow on a diagonal is shown in Figure \ref{fig:traj1}.

\begin{figure}[ht]
\begin{center}
\begin{tikzpicture}[decoration = {
            markings,
            mark = at position 0.5 with {\arrow{triangle 45}},
        }]
    \pgftransformxslant{0.3}
    \pgftransformscale{2}
    \coordinate (A) at (1, 0);
    \coordinate (B) at (2.25, 0);
    \coordinate (C) at (2.25, 2);
    \coordinate (D) at (0, 2);
    \coordinate (E) at (0, 1);
    \coordinate (X) at (1.75, 1.2);
    \coordinate (T) at (1.183, 0.292);
    \coordinate (T2) at (1.846, 1.354);
    
    \fill[fill = blue!5] (A) -- (B) -- (C) -- (D) -- (E) -- (A);
    
    \draw [very thick, color = blue] 
        (B) -- (C) node[pos = 0.5, right, color = black]{$A$}
        (D) -- (E) node[pos = 0.5, left, color = black]{$A$};
    \draw [very thick, color = red] 
        (A) -- (B) node[pos = 0.5, below, color = black]{$B$}
        (C) -- (D) node[pos = 0.5, above, color = black]{$B$};
    \draw [very thick] (E) -- (A);
    \draw [thick] (A) -- (C);

    \draw[thick, postaction = decorate] (X) -- (0.95, 2);
    \draw[thick] (1.475, 0) -- (T);
    \draw[thick, postaction = decorate] (T) -- (0, 1.475);
    \draw[thick] (2.25, 0.95) -- (T2);

    \pgftransformxslant{-0.3}
    
    \filldraw (A) circle (1pt)
    (B) circle (1pt)
    (C) circle (1pt)
    (D) circle (1pt)
    (E) circle (1pt);
    
    \filldraw
        (X) circle (1pt)
        (T) circle (1pt)
        (T2) circle (1pt);
    
    \node[black] at (2.1, 1.05) {$x$};
    \node[black] at (1.6, 0.292) {$T(x)$};
    \node[black] at (2.2, 1.65) {$T^2(x)$};
        
    \pgftransformscale{0.5}
    \pgftransformshift{\pgfpoint{7cm}{1.5cm}}
    \coordinate (A) at (0, 1);
    \coordinate (B) at (2.5, 1);
    \coordinate (C) at (7, 1);
    \coordinate (D) at (0, 0);
    \coordinate (E) at (4, 0);
    \coordinate (F) at (7, 0);
    
    \draw [very thick, color = blue] (A) -- (B) node[pos = 0.5, above, color = black]{$A$};
    \draw [very thick, color = red] (B) -- (C) node[pos = 0.5, above, color = black]{$B$};
    \draw [very thick, color = blue] (E) -- (F) node[pos = 0.5, above, color = black]{$T(A)$};
    \draw [very thick, color = red] (D) -- (E) node[pos = 0.5, above, color = black]{$T(B)$};
    
    \filldraw [color = black] (B) circle (2pt);
    \filldraw [color = black] (E) circle (2pt);
\end{tikzpicture}
\caption{The flow and first return map in the direction parallel to the boundary.}
\label{fig:traj1}
\end{center}
\end{figure}
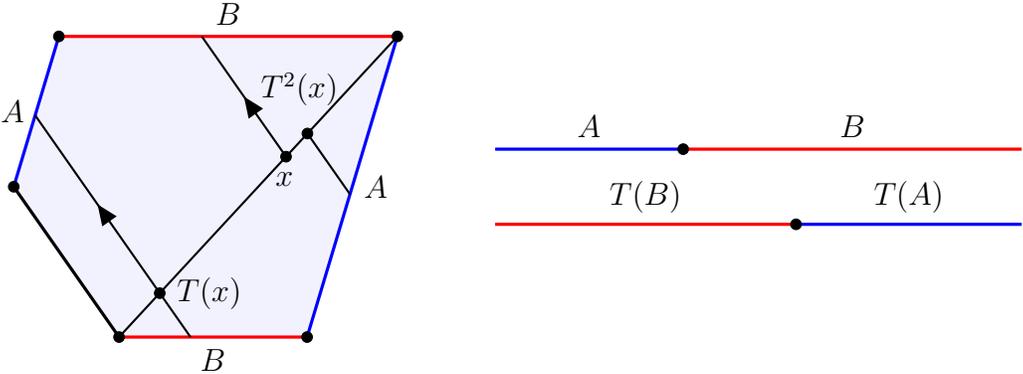

We find that the first return map is a bijective 2-AIET.  
In \cite{invariantmeasuresAIETs}, a formula is provided to compute the rotation number of such a map.
From the rotation number one can determine the dynamics.  
If the rotation number is irrational, then the dynamics are minimal.  
If the rotation number is rational, the dynamics may be either completely periodic or Morse-Smale.  
Once the dynamical behavior of $T$ is understood, the directional flow on $X$ is a suspension of this map and shares the same behavior.

\subsection{Directions in $I_1$ and $I_5$}
\label{sec:I1directions}

The first return maps on some transverse diagonal to directions in $I_1$ are given by non-surjective 2-AIETs of a specific form. These are maps which swap the two intervals, expand one and contract the other, and place them against the left and right endpoints of the codomain.  At first this desired form is not immediately apparent, but we may attain this form after applying the first return map once and restricting to the image.  See Figure $\ref{fig:traj2}$ for an example of this discussion.

\begin{figure}[ht]
\begin{center}
\begin{tikzpicture}[decoration = {
            markings,
            mark = at position 0.5 with {\arrow{triangle 45}},
        }]
    \pgftransformxslant{0.3}
    \pgftransformscale{2}
    \coordinate (A) at (1, 0);
    \coordinate (B) at (2, 0);
    \coordinate (C) at (2, 2);
    \coordinate (D) at (0, 2);
    \coordinate (E) at (0, 1);
    \coordinate (X) at (0.25, 1.125);
    \coordinate (T) at (1.7, 1.85);
    \coordinate (T2) at (1.05, 1.525);
    
    \fill[fill = blue!5] (A) -- (B) -- (C) -- (D) -- (E) -- (A);
    
    \draw [very thick, color = blue] 
        (B) -- (C) node[pos = 0.5, right, color = black]{$A$}
        (D) -- (E) node[pos = 0.5, left, color = black]{$A$};
    \draw [very thick, color = red] 
        (A) -- (B) node[pos = 0.5, below, color = black]{$B$}
        (C) -- (D) node[pos = 0.5, above, color = black]{$B$};
    \draw [very thick] (E) -- (A);
    \draw [thick] (E) -- (C);

    \draw[thick, postaction = decorate] (X) -- (0, 1.625);
    \draw[thick, postaction = decorate] (2, 1.25) -- (2-0.375, 2);
    \draw[thick, postaction = decorate] (1.8125, 0) -- (T2);
    
    \pgftransformxslant{-0.3}
    
    \filldraw (A) circle (1pt)
    (B) circle (1pt)
    (C) circle (1pt)
    (D) circle (1pt)
    (E) circle (1pt);
    
    \filldraw
        (X) circle (1pt)
        (T) circle (1pt)
        (T2) circle (1pt);
    
    \node[black] at (0.6, 0.95) {$x$};
    \node[black] at (2.2, 2.12) {$T(x)$};
    \node[black] at (1.4, 1.75) {$T^2(x)$};
        
    \pgftransformscale{0.5}
    \pgftransformshift{\pgfpoint{7cm}{3cm}}
    \coordinate (A) at (0, 1);
    \coordinate (B) at (2, 1);
    \coordinate (C) at (7, 1);
    \coordinate (D) at (0, 0);
    \coordinate (E) at (1, 0);
    \coordinate (F) at (4, 0);
    \coordinate (G) at (7, 0);
    
    \draw [very thick, color = blue] (A) -- (B) node[pos = 0.3, above, color = black]{$A$};
    \draw [very thick, color = red] (B) -- (C) node[pos = 0.5, above, color = black]{$B$};
    \draw [very thick, color = blue] (F) -- (G) node[pos = 0.5, above, color = black]{$T(A)$};
    \draw [very thick, color = red] (E) -- (F) node[pos = 0.5, above, color = black]{$T(B)$};
    \draw [very thick, dashed] (1, -0.5) -- (1, 1.5);
    
    \filldraw [color = black] (B) circle (2pt);
    \filldraw [color = black] (F) circle (2pt);

    \pgftransformshift{\pgfpoint{0cm}{-3cm}}
    \draw[very thick, ->] (3.5, 2.35) -- (3.5, 1.65);
    \coordinate (A) at (1, 1);
    \coordinate (B) at (2, 1);
    \coordinate (C) at (7, 1);
    \coordinate (D) at (1, 0);
    \coordinate (E) at (4, 0);
    \coordinate (F) at (5, 0);
    \coordinate (G) at (7, 0);
    
    \draw [very thick, color = blue] (A) -- (B) node[pos = 0.5, above, color = black]{$A$};
    \draw [very thick, color = red] (B) -- (C) node[pos = 0.5, above, color = black]{$B$};
    \draw [very thick, color = blue] (F) -- (G) node[pos = 0.5, above, color = black]{$T(A)$};
    \draw [very thick, color = red] (D) -- (E) node[pos = 0.5, above, color = black]{$T(B)$};
    
    \filldraw [color = black] (B) circle (2pt);
\end{tikzpicture}
\caption{An example of a flow and first return map in the first interval of directions.}
\label{fig:traj2}
\end{center}
\end{figure}
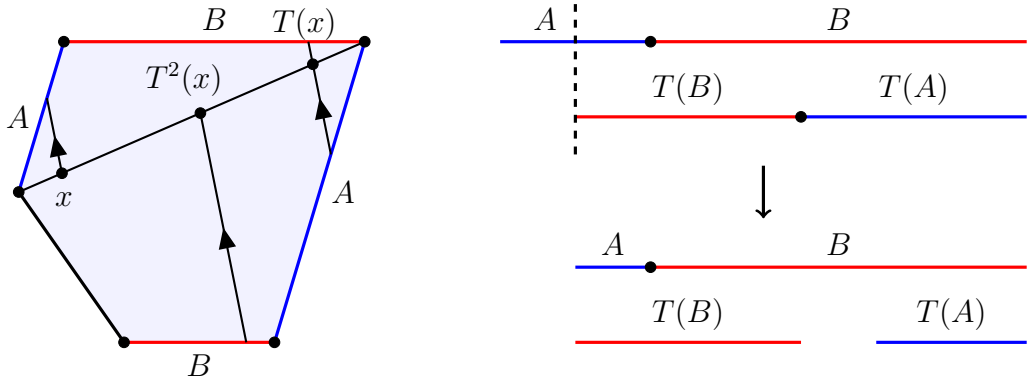

The dynamics of maps which look like this will be studied in Section $\ref{sec:pA>1}$.  We note here that maps in $I_5$ are of the same form as those in $I_1$, except for the choice of which interval expands and which contracts.  By reflecting, we will see that the dynamics of maps in $I_5$ will be the mirror image of the dynamics of maps in $I_1$.

\subsection{Directions in $I_2$ and $I_4$}
\label{sec:I2directions}

Trajectories moving in directions from $I_2$ are caught within an immediately present subsurface called a \emph{dilation cylinder}, an affine cylinder with holonomy across its core curve. 

There are two obvious dilation cylinders present in the polygonal model for $X$. 
The subsurface which consists of the convex hull of the two edges labeled $B$ in the polygonal model is the dilation cylinder corresponding to $I_2$.

This dilation cylinder is seen in the first return map. In Figure $\ref{fig:traj3}$, it is seen that flows in $I_2$ after one iteration only land on the red interval $B$.  The first return map is an injective affine contraction, so it has a unique fixed point.

This tells that the forward orbit from any noncritical point moving in a direction in $I_2$ is attracted to a closed orbit in the dilation cylinder.
Note also that at each boundary direction of $I_2$, the directional foliation is attracted to a saddle connection.

Directions in $I_4$ obey similar behavior, captured in the dilation cylinder which is the convex hull of the two edges labeled $A$.  We conclude that every directional foliation in the interiors of $I_2$ and $I_4$ converges to a closed orbit.

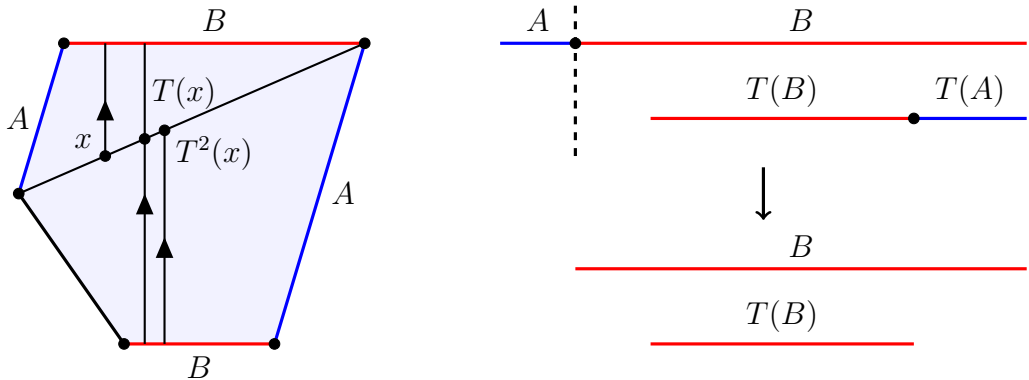
\begin{figure}[ht]
\begin{center}
\begin{tikzpicture}[decoration = {
            markings,
            mark = at position 0.5 with {\arrow{triangle 45}},
        }]
    \pgftransformxslant{0.3}
    \pgftransformscale{2}
    \coordinate (A) at (1, 0);
    \coordinate (B) at (2, 0);
    \coordinate (C) at (2, 2);
    \coordinate (D) at (0, 2);
    \coordinate (E) at (0, 1);
    \coordinate (X) at (0.5, 1.25);
    \coordinate (T) at (0.728, 1.364);
    \coordinate (T2) at (0.843, 1.421);
    
    \fill[fill = blue!5] (A) -- (B) -- (C) -- (D) -- (E) -- (A);
    
    \draw [very thick, color = blue] 
        (B) -- (C) node[pos = 0.5, right, color = black]{$A$}
        (D) -- (E) node[pos = 0.5, left, color = black]{$A$};
    \draw [very thick, color = red] 
        (A) -- (B) node[pos = 0.5, below, color = black]{$B$}
        (C) -- (D) node[pos = 0.5, above, color = black]{$B$};
    \draw [very thick] (E) -- (A);
    \draw [thick] (E) -- (C);
    
    \pgftransformxslant{-0.3}
    
    \filldraw (A) circle (1pt)
    (B) circle (1pt)
    (C) circle (1pt)
    (D) circle (1pt)
    (E) circle (1pt);
    
    \filldraw
        (X) circle (1pt)
        (T) circle (1pt)
        (T2) circle (1pt);
    
    \node[black] at (0.73, 1.35) {$x$};
    \node[black] at (1.4, 1.67) {$T(x)$};
    \node[black] at (1.6, 1.27) {$T^2(x)$};

    \draw[thick, postaction = decorate] (X) -- (0.875, 2);
    \draw[thick, postaction = decorate] (1.1375, 0) -- (1.1375, 2);
    \draw[thick, postaction = decorate] (1.269, 0) -- (T2);

    \pgftransformscale{0.5}
    \pgftransformshift{\pgfpoint{7cm}{3cm}}
    \coordinate (A) at (0, 1);
    \coordinate (B) at (1, 1);
    \coordinate (C) at (7, 1);
    \coordinate (D) at (0, 0);
    \coordinate (E) at (2, 0);
    \coordinate (F) at (5.5, 0);
    \coordinate (G) at (7, 0);
    
    \draw [very thick, color = blue] (A) -- (B) node[pos = 0.5, above, color = black]{$A$};
    \draw [very thick, color = red] (B) -- (C) node[pos = 0.5, above, color = black]{$B$};
    \draw [very thick, color = blue] (F) -- (G) node[pos = 0.5, above, color = black]{$T(A)$};
    \draw [very thick, color = red] (E) -- (F) node[pos = 0.5, above, color = black]{$T(B)$};
    \draw [very thick, dashed] (1, -0.5) -- (1, 1.5);
    
    \filldraw [color = black] (B) circle (2pt);
    \filldraw [color = black] (F) circle (2pt);

    \pgftransformshift{\pgfpoint{0cm}{-3cm}}
    \draw[very thick, ->] (3.5, 2.35) -- (3.5, 1.65);
    \coordinate (B) at (1, 1);
    \coordinate (C) at (7, 1);
    \coordinate (D) at (2, 0);
    \coordinate (E) at (5.5, 0);
    
    \draw [very thick, color = red] (B) -- (C) node[pos = 0.5, above, color = black]{$B$};
    \draw [very thick, color = red] (D) -- (E) node[pos = 0.5, above, color = black]{$T(B)$};
    
\end{tikzpicture}
\caption{An example of a flow and first return map in the second interval of directions. The lower map is a restriction of the upper map to part of the diagonal.}
\label{fig:traj3}
\end{center}
\end{figure}

\subsection{Directions in $I_3$}
\label{sec:I3directions}

The directions in $I_3$ have a similar structure to those in $I_1$ and $I_5$, in the sense that their first return maps are injective 2-AIETs.  They swap the two intervals, contracting both, and place them against the left and right endpoints of the codomain.  These maps will immediately have the desired structure without needing to restrict to the image.  See Figure \ref{fig:traj4}.

\begin{figure}[ht]
\begin{center}
\begin{tikzpicture}[decoration = {
            markings,
            mark = at position 0.8 with {\arrow{triangle 45}},
        }]
    \pgftransformxslant{0.3}
    \pgftransformscale{2}
    \coordinate (A) at (1, 0);
    \coordinate (B) at (2, 0);
    \coordinate (C) at (2, 2);
    \coordinate (D) at (0, 2);
    \coordinate (E) at (0, 1);
    \coordinate (X) at (0.5, 1.5);
    \coordinate (T) at (1.75, 0.25);
    \coordinate (T2) at (0.375, 1.625);
    
    \fill[fill = blue!5] (A) -- (B) -- (C) -- (D) -- (E) -- (A);
    
    \draw [very thick, color = blue] 
        (B) -- (C) node[pos = 0.5, right, color = black]{$A$}
        (D) -- (E) node[pos = 0.5, left, color = black]{$A$};
    \draw [very thick, color = red] 
        (A) -- (B) node[pos = 0.5, below, color = black]{$B$}
        (C) -- (D) node[pos = 0.5, above, color = black]{$B$};
    \draw [very thick] (E) -- (A);
    \draw [thick] (B) -- (D);

    \draw[thick, postaction = decorate] (X) -- (1, 2);
    \draw[thick, postaction = decorate] (1.5, 0) -- (T);
    \draw[thick] (T) -- (2, 0.5);
    \draw[thick, postaction = decorate] (0, 1.25) -- (T2);

    \pgftransformxslant{-0.3}
    
    \filldraw (A) circle (1pt)
    (B) circle (1pt)
    (C) circle (1pt)
    (D) circle (1pt)
    (E) circle (1pt);
    
    \filldraw
        (X) circle (1pt)
        (T) circle (1pt)
        (T2) circle (1pt);
    
    \node[black] at (0.9, 1.35) {$x$};
    \node[black] at (1.52, 0.29) {$T(x)$};
    \node[black] at (1.02, 1.85) {$T^2(x)$};

    \pgftransformscale{0.5}
    \pgftransformshift{\pgfpoint{7cm}{1.5cm}}
    \coordinate (A) at (0, 1);
    \coordinate (B) at (4, 1);
    \coordinate (C) at (7, 1);
    \coordinate (D) at (0, 0);
    \coordinate (E) at (2, 0);
    \coordinate (F) at (5.5, 0);
    \coordinate (G) at (7, 0);
    
    \draw [very thick, color = blue] (B) -- (C) node[pos = 0.5, above, color = black]{$A$};
    \draw [very thick, color = red] (A) -- (B) node[pos = 0.5, above, color = black]{$B$};
    \draw [very thick, color = blue] (D) -- (E) node[pos = 0.5, above, color = black]{$T(A)$};
    \draw [very thick, color = red] (G) -- (F) node[pos = 0.5, above, color = black]{$T(B)$};
    
    \filldraw [color = black] (B) circle (2pt);
    
\end{tikzpicture}
\caption{An example of a flow and first return map in the third interval of directions.}
\label{fig:traj4}
\end{center}
\end{figure}
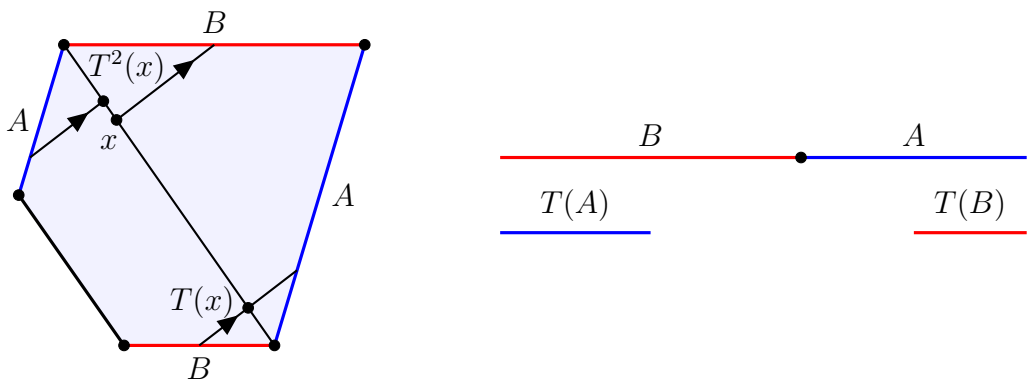

We call maps like the first return maps in the intervals $I_1$, $I_3$, and $I_5$ \emph{$(\rho_A, \rho_B)$-maps}.  These maps will be defined (and their name justified) in the next section.  While they are similar in structure to their surjective cousins, found in the boundary-parallel direction, their non-surjectivity leads to wildly different dynamics.  To understand the dynamics of flows on dilation tori, it is essential that we analyze the dynamics of these maps in detail.

\section{Dynamics of $(\rho_A, \rho_B)$-maps}
\label{sec:dynamics}

The main goal of this section is to prove the following theorem. 

\cantor

We will do this by analyzing a family of maps called $(\rho_A, \rho_B)$-maps. We consider two dilation factors $\rho_A, \rho_B > 0$.  We use an interval $D = [0, r]$ as our domain.
By a $(\rho_A, \rho_B)$-map, we mean an injective $2$-AIET $T: D \to D$ with the following properties:
\begin{itemize}
    \itemsep0em
    \item 
    $D$ is broken into two intervals $A$ and $B$, divided by the point $x_T \in D$.
    \item 
    On $A$, the derivative of $T$ is $\rho_A$.
    \item 
    On $B$, the derivative of $T$ is $\rho_B$.
    \item 
    $T$ moves $A$ to the right end of $D$, so the right endpoint of $T(A)$ is $r$.
    \item 
    $T$ sends $B$ to the left end of $D$, so the left endpoint of $T(B)$ is $0$.
\end{itemize}
The final first return maps in Figures \ref{fig:traj2} and \ref{fig:traj4} are examples of $(\rho_A,\rho_B)$-maps.

For fixed $\rho_A, \rho_B$, let $\mathcal{I}(\rho_A, \rho_B)$ be the space of $(\rho_A, \rho_B)$-maps.  We identify maps up to scaling, and typically represent these equivalence classes of maps by scaling the domain to be $D = [0, 1]$.
The space $\mathcal{I}(\rho_A, \rho_B)$ can then be parameterized by the point of discontinuity $x_T \in D$.
When $\rho_A < 1$, $\rho_B < 1$ so that both intervals contract, the discontinuities $x_T$ may run over the whole interval $D$ without violating injectivity.
If one of the intervals is expanding, a subinterval of $D$ contains the discontinuity points which parameterize injective maps in $\mathcal{I}(\rho_A, \rho_B)$ (this case is studied in Section \ref{sec:pA>1}).  
Any $(\rho_A, \rho_B)$-map $T$ can be represented by a diagram like the one in Figure \ref{fig:pApBex}, where $T$ maps the top intervals to the bottom ones:

\begin{figure}[ht]
\begin{center}
\begin{tikzpicture}
  \coordinate (A) at (0, 1.2);
  \coordinate (B) at (4.5, 1.2);
  \coordinate (C) at (7, 1.2);
  \coordinate (D) at (0, 0);
  \coordinate (E) at (2, 0);
  \coordinate (F) at (4.1, 0);
  \coordinate (G) at (7, 0);

  \draw [very thick, color = blue] (A) -- (B) 
    node[pos = 0.5, above, color = black]{$A$} 
    node[pos = 0.5, below, color = black]{$x_T$};
  \draw [very thick, color = red] (B) -- (C) 
    node[pos = 0.5, above, color = black]{$B$} 
    node[pos = 0.5, below, color = black]{$1 - x_T$};
  \draw [very thick, color = blue] (F) -- (G) 
    node[pos = 0.5, above, color = black]{$T(A)$} 
    node[pos = 0.5, below, color = black]{$\rho_A x_T$};
  \draw [very thick, color = red] (D) -- (E) 
    node[pos = 0.5, above, color = black]{$T(B)$} 
    node[pos = 0.5, below, color = black]{$\rho_B(1 - x_T)$};
  
  \filldraw [color = black] (B) circle (2pt) node[above]{$x_T$};
\end{tikzpicture}
\caption{A $(\rho_A, \rho_B)$-map with discontinuity point $x_T$.}
\label{fig:pApBex}
\end{center}
\end{figure}
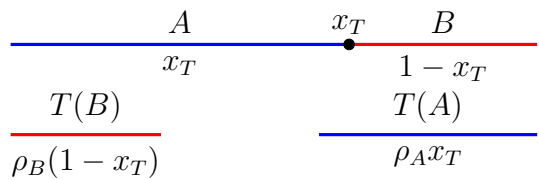

In order to understand the long-term dynamics of $(\rho_A, \rho_B)$-maps, we introduce an algorithm called \emph{Rauzy-Veech induction} which exposes the dynamics of these maps.  This algorithm is the key ingredient in our proof of Theorem \ref{thm:1}.

\subsection{Rauzy-Veech induction on $(\rho_A, \rho_B)$-maps}
\label{sec:pApBInduction}

Our main tool in analyzing the dynamics of $(\rho_A, \rho_B)$-maps will be a form of Rauzy-Veech induction. Rauzy-Veech induction (often shortened as Rauzy induction) was originally defined for interval exchange transformations (IETs) and is a way of accelerating the dynamics of these maps. This way, we can keep track of the trajectory of a point in the original IET if we understand its trajectory in the accelerated map. For an introduction to Rauzy induction on IETs, see \cite{yoccoz}. 

The Rauzy induction algorithm has also been used by various authors (e.g. \cite{mmy}) to study surjective AIETs, and in \cite{ohdTori}, Boulanger and Ghazouani define a version of Rauzy induction for nonsurjective 2-AIETs that we will briefly summarize here.

Let $T: D \to D$ be a $(\rho_A, \rho_B)$-map. We will then consider the first return map of the AIET to a carefully chosen subinterval via a \emph{Rauzy induction step}. One step of Rauzy induction produces another $(\rho_A^\prime, \rho_B^\prime)$-map where $\rho_A^\prime$ and $\rho_B^\prime$ are products of the original $\rho_A$ and $\rho_B$. Multiple steps of Rauzy induction starting with a map in $\mathcal{I}(\rho_A, \rho_B)$ then give us a collection of maps in the spaces $\mathcal{I}(\rho_A^{m_A}\rho_B^{m_B}, \rho_A^{n_A}\rho_B^{n_B})$ for integers $m_A, m_B, n_A, n_B \geq 0$.  This leads us to the following inductive definition of Rauzy induction:

Suppose $T \in \mathcal{I}(\rho_A^{m_A}\rho_B^{m_B}, \rho_A^{n_A}\rho_B^{n_B})$, let $A$ and $B$ be the left and right top intervals of $T$ respectively, 
and let $\lambda_A, \lambda_B$ be their lengths.  
There are three possibilities:

\begin{enumerate}
    \item[(1)]
    \begin{enumerate}
    
        \item[(R)] (Right Rauzy Induction) 
        $B \subsetneq T(A)$, i.e. $\lambda_B<\rho_A^{m_A}\rho_B^{m_B}\lambda_A$
        \\
        Letting $D' = D \setminus B$, we consider the first return map $T'$ on $D'$. 
        Then:\\
        $T' \in \mathcal{I}(\rho_A^{m_A}\rho_B^{m_B}, \rho_A^{m_A + n_A}\rho_B^{m_B + n_B})$ (recall that we may rescale the domain)\\
        and the lengths become 
        $\lambda'_A := \lambda_A - \rho_A^{-m_A}\rho_B^{-m_B} \lambda_B$ 
        and $\lambda'_B := \rho_A^{-m_A}\rho_B^{-m_B}\lambda_B$.\\
        This corresponds to the transformation $\lambda' = R_{m_A, m_B} \lambda$ where 
        \[
            \lambda := \begin{pmatrix} \lambda_A \\ \lambda_B \end{pmatrix},
            \qquad \lambda' := \begin{pmatrix} \lambda'_A \\ \lambda'_B \end{pmatrix}, \qquad R_{m_A, m_B} := \begin{pmatrix} 1 & -\rho_A^{-m_A}\rho_B^{-m_B} \\ 0 & \rho_A^{-m_A}\rho_B^{-m_B} \end{pmatrix}
        \]
        
        An example of right Rauzy induction is shown in Figure \ref{fig:rightInduction}.
        
        \begin{figure}[ht]
        \begin{center}
        \begin{tikzpicture}
          \draw [very thick, color = blue] (0, 1) -- (3.5, 1) node[pos = 0.5, above, color = black]{$A$};
          \draw [very thick, color = red] (3.5, 1) -- (5, 1) node[pos = 0.5, above, color = black]{$B$};
          \draw [very thick, color = blue] (2.5, 0) -- (5, 0) node[pos = 0.5, above, color = black]{$A$};
          \draw [very thick, color = red] (0, 0) -- (1, 0) node[pos = 0.5, above, color = black]{$B$};
          \draw [dashed] (3.5, -0.25) -- (3.5, 1.5);
          
          \draw[->, very thick] (5.75, 0.5) -- (6.25, 0.5);
          
          \draw [very thick, color = blue] (7, 1) -- (8.4, 1) node[pos = 0.5, above, color = black]{$A$};
          \draw [very thick, color = red] (8.4, 1) -- (10.5, 1) node[pos = 0.5, above, color = black]{$B$};
          \draw [very thick, color = blue] (9.5, 0) -- (10.5, 0) node[pos = 0.5, above, color = black]{$A$};
          \draw [very thick, color = red] (7, 0) -- (8, 0) node[pos = 0.5, above, color = black]{$B$};
          
          \draw[fill] (3.5, 1) circle (2pt);
          \draw[fill] (8.4,1) circle (2pt)
            (8.4,1) circle (2pt);
        \end{tikzpicture}
        \caption{A right Rauzy induction step.  $T$ is on the left, and the result $T'$ on the right.}
        \label{fig:rightInduction}
        \end{center}
        \end{figure}

        \item[(b)] (Left Rauzy Induction) 
        $A \subsetneq T(B)$, i.e. $\lambda_A < \rho_A^{n_A}\rho_B^{n_B} \lambda_B$
        
        Letting $D' = D \setminus A$, we consider the first return map $T'$ on $D'$.  Then: \\
        $T' \in \mathcal{I}(\rho_A^{m_A + n_A}\rho_B^{m_B+n_B}, 
        \rho_A^{n_A}\rho_B^{n_B})$ and the new lengths are $\lambda' = L_{n_A, n_B} \lambda$ where
        \[
            \lambda := \begin{pmatrix} \lambda_A \\ \lambda_B \end{pmatrix}, \qquad \lambda' := \begin{pmatrix} \lambda'_A \\ \lambda'_B \end{pmatrix}, \qquad L_{n_A, n_B} := \begin{pmatrix} \rho_A^{-n_A}\rho_B^{-n_B} & 0 \\ -\rho_A^{-n_A}\rho_B^{-n_B} & 1 \end{pmatrix}
        \]
        
        An example of a left Rauzy induction step is shown in Figure \ref{fig:leftInduction}.
        
        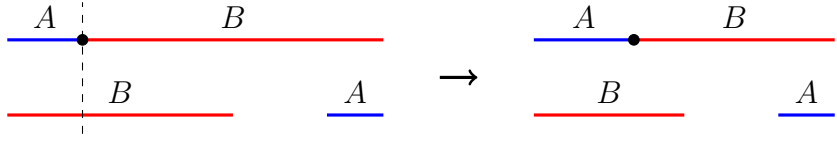
\begin{figure}[ht]
        \begin{center}
        \begin{tikzpicture}
          \draw [very thick, color = blue] (0, 1) -- (1, 1) node[pos = 0.5, above, color = black]{$A$};
          \draw [very thick, color = red] (1, 1) -- (5, 1) node[pos = 0.5, above, color = black]{$B$};
          \draw [very thick, color = blue] (4.25, 0) -- (5, 0) node[pos = 0.5, above, color = black]{$A$};
          \draw [very thick, color = red] (0, 0) -- (3, 0) node[pos = 0.5, above, color = black]{$B$};
          \draw [dashed] (1, -0.25) -- (1, 1.5);
          
          \draw[->, very thick] (5.75, 0.5) -- (6.25, 0.5);
          
          \draw [very thick, color = blue] (7, 1) -- (8.33, 1) node[pos = 0.5, above, color = black]{$A$};
          \draw [very thick, color = red] (8.33, 1) -- (11, 1) node[pos = 0.5, above, color = black]{$B$};
          \draw [very thick, color = blue] (10.25, 0) -- (11, 0) node[pos = 0.5, above, color = black]{$A$};
          \draw [very thick, color = red] (7, 0) -- (9, 0) node[pos = 0.5, above, color = black]{$B$};
          
          \draw[fill] (1, 1) circle (2pt);
          \draw[fill] (8.33,1) circle (2pt);
        \end{tikzpicture}
        \caption{A left Rauzy induction step.  $T$ is on the left, and the result $T'$ is on the right.}
        \label{fig:leftInduction}
        \end{center}
        \end{figure}
        
    \end{enumerate}

    \item[(2)] (Termination)  
    $T(A) \subseteq B$ and $T(B) \subseteq A$, i.e. 
    $\lambda_A \geq \rho_A^{n_A}\rho_B^{n_B} \lambda_B$ and $\lambda_B \geq \rho_A^{m_A}\rho_B^{m_B} \lambda_A$
    
    Letting $D' = D \setminus T(A)$, we consider the first return map $T'$ on $D'$.  
    The resulting map has image in $A$, and is a contraction in $A$, 
    so orbits under the map converge to an attractive fixed point in $A$ of derivative $\rho_A^{m_A + n_A}\rho_B^{m_B + n_B}$.  
    The picture is in Figure \ref{fig:terminateInduction}.
    
    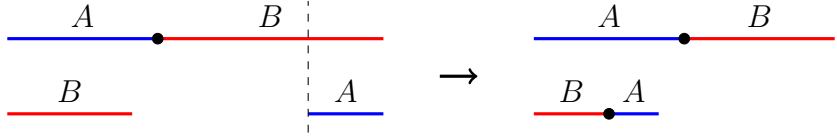
\begin{figure}[ht]
    \begin{center}
        \begin{tikzpicture}
          \draw [very thick, color = blue] (0, 1) -- (2, 1) node[pos = 0.5, above, color = black]{$A$};
          \draw [very thick, color = red] (2, 1) -- (5, 1) node[pos = 0.5, above, color = black]{$B$};
          \draw [very thick, color = blue] (4, 0) -- (5, 0) node[pos = 0.5, above, color = black]{$A$};
          \draw [very thick, color = red] (0, 0) -- (1.66, 0) node[pos = 0.5, above, color = black]{$B$};
          \draw [dashed] (4, -0.25) -- (4, 1.5);
          
          \draw[fill] (2,1) circle (2pt);
          
          \draw[->, very thick] (5.75, 0.5) -- (6.25, 0.5);
          
          \draw [very thick, color = blue] (7, 1) -- (9, 1) node[pos = 0.5, above, color = black]{$A$};
          \draw [very thick, color = red] (9, 1) -- (11, 1) node[pos = 0.5, above, color = black]{$B$};
          \draw [very thick, color = blue] (8, 0) -- (8.66, 0) node[pos = 0.5, above, color = black]{$A$};
          \draw [very thick, color = red] (7, 0) -- (8, 0) node[pos = 0.5, above, color = black]{$B$};
          
          \draw[fill] (9,1) circle (2pt);
          \draw[fill] (8,0) circle (2pt);
        \end{tikzpicture}
    \caption{An example of termination.  The resulting map has an attractive fixed point.}
    \label{fig:terminateInduction}
    \end{center}
    \end{figure}
    
\end{enumerate}

\subsection{Understanding when the Rauzy algorithm terminates}
\label{sec:outlineTermination}

Our goal now is to understand when the process of iteratively applying Rauzy induction to our original $(\rho_A, \rho_B)$-map terminates after finitely many steps, since terminating maps have an attracting \emph{periodic orbit} or \emph{critical orbit} (orbit that eventually hits the break point $x_T$). 

We begin with an empty word $w$ and a map $T \in \mathcal{I}(\rho_A, \rho_B)$.  
If a Rauzy induction step is possible, we apply the appropriate (R) or (L) induction and record the letter $R$ or $L$ at the end of $w$. We repeat on the resulting maps either infinitely or until we reach a termination step. 

We assume that our initial $(\rho_A, \rho_B)$ map $T$ is a map of the unit interval $[0,1]$ to itself. We then parameterize $(\rho_A, \rho_B)$-maps by the position of their breakpoint $x_T$ in $[0,1]$. For a finite word $w$, we let $I(w) \subset [0,1]$ be the set of maps that follow the word $w$ in their sequence of left and right induction steps. We also let $H(w) \subset I(w)$ be the set of maps that follow word $w$ and whose next step is a termination step. 

The outline of our argument will go as follows: 

\begin{enumerate}
    \item We will first consider the contracting case when $\rho_A, \rho_B < 1$. In this case, we will explicitly find the intervals of discontinuity points $H(w) \subset I(w) \subset [0,1]$ parameterizing those maps in $\mathcal{I}(\rho_A, \rho_B)$ where the Rauzy induction steps follow the sequence of inductions given by the word $w$, or terminate after that sequence.
    \item Building off of our work in the previous step, we will find that $|H(w)|/|I(w)|$ is uniformly bounded below.  Since the set of parameters $H(w)$ where the Rauzy induction terminates is constructed like the complement of the middle thirds Cantor set, this bound will show that the terminating parameters form a full measure subset of the interval whose complement is a Cantor set. 
    
    \item Finally, we will move on to the case when one of $\rho_A$ or $\rho_B$ is greater than $1$. We will find that with the exception of one starting parameter, the Rauzy induction will eventually land us in a terminating case or in a case where our scaling factors $\rho'_A$ and $\rho'_B$ are both contracting. We can then use our work from the $\rho_A, \rho_B < 1$ case to again conclude that the terminating parameters form a full measure subset of the interval that is the complement of a Cantor set.
    \item Using the results in steps 2 and 3, we can lift the attracting orbits on these $(\rho_A, \rho_B)$-maps to attracting orbits on corresponding one-holed dilation tori to provide a proof of Theorem \ref{thm:1}.
\end{enumerate}

\subsection{The contracting case}
\label{sec:pApB<1}

We start with the case where $\rho_A, \rho_B < 1$.  
Let $T \in \mathcal{I}(\rho_A, \rho_B)$ where we have lengths $\lambda_A = x_T$, 
$\lambda_B = 1 - x_T$ for $x_T \in (0, 1)$.  
We can consider a word $w$ of length $\ell$, $w = w_1\cdots w_\ell$ with $w_i \in \{R, L\}$, encoding a sequence of Rauzy induction moves.  
We associate to $w$ sequences of integers $m_A(i), m_B(i), n_A(i), n_B(i)$ with 
$(m_A(0), m_B(0), n_A(0), n_B(0)) = (1, 0, 0, 1)$ and the recurrence
\begin{align*}
    (m_A(i + 1), & m_B(i + 1),  n_A(i + 1), n_B(i + 1)) \\
    & = \begin{cases} (m_A(i), m_B(i), m_A(i) + n_A(i), m_B(i) + n_B(i)) & w_{i + 1} = R \\
    (m_A(i) + n_A(i), m_B(i) + n_B(i), n_A(i), n_B(i))& w_{i + 1} = L. \end{cases}
\end{align*}

Let $(m_A, m_B, n_A, n_B) := (m_A(\ell), m_B(\ell), n_A(\ell), n_B(\ell))$ be the final term.  
Then upon applying the Rauzy inductions in $w$ to $T$, the resulting map is in $\mathcal{I}(\rho_A^{m_A}\rho_B^{m_B}, \rho_A^{n_A}\rho_B^{n_B})$.  
We can define a sequence of length transformation matrices $M_i$ with $M_0 = I$ and
\[
    M_{i} = \begin{cases} R_{m_A(i), m_B(i)} M_{i - 1} & w_i = R \\ 
    L_{n_A(i), n_B(i)} M_{i - 1} & w_i = L, \end{cases}
\]
where the $R$ and $L$ matrices are defined as in Section \ref{sec:pApBInduction}, so that the lengths of the final intervals up to scaling are $\begin{pmatrix} \lambda_A' \\ \lambda_B' \end{pmatrix} 
:= M_\ell \begin{pmatrix} \lambda_A \\ \lambda_B \end{pmatrix} 
= M_\ell \begin{pmatrix} x_T \\ 1 - x_T \end{pmatrix}$.

Note that the sequence of moves $w$ is the valid sequence of Rauzy inductions executed on $T$ if and only if $\lambda'_A, \lambda'_B \geq 0$.  
Indeed, if the sequence is invalid then the result of some step will give a negative entry in the length vector.  
Once one entry of the length vector is negative, it is easy to see that further multiplications by the $L$ and $R$ matrices keep some entry negative.

Suppose the final matrix is $M_\ell =: \begin{pmatrix} a & b \\ c & d \end{pmatrix}$.  The following lemma will help us identify the parameters corresponding to a particular set of Rauzy induction moves. 
\begin{lemma}
\label{lem:8}
    For any word $w \in \{R, L\}^\ell$, 
    the corresponding length transformation matrix $M_\ell$ has $a, d \geq 0$ and $b, c \leq 0$.  Also, $\det M_\ell \geq 0$.  
    As a consequence, $0 \leq \frac{-b}{a - b} \leq \frac{d}{d - c} \leq 1$.
\end{lemma}
\begin{proof}
    We will show this by induction. We first note that the lemma is true for $M_0 = I$. Furthermore, all $R_{m_A, m_B}$ and $L_{n_A, n_B}$ have positive determinant, so $\det M_\ell \geq 0$.  
    \\
    Now we suppose the lemma holds for $M_i = \begin{pmatrix} a & b \\ c & d \end{pmatrix}$.  
    Then $R_{m_A, m_B}M_i$ has diagonals $d\rho_A^{-m_A}\rho_B^{-m_B} \geq 0$ 
    and $a - c\rho_A^{-m_A}\rho_B^{-m_B} \geq 0$, and off-diagonal entries $b - d\rho_A^{-m_A}\rho_B^{-m_B} \leq 0$ and $c\rho_A^{-m_A}\rho_B^{-m_B} \leq 0$.
    A similar computation shows the result for $L_{n_A, n_B}M_i$, and the result follows by induction.  These inequalities are enough to show $0 \leq \frac{-b}{a - b} \leq \frac{d}{d - c} \leq 1$.
\end{proof}

This lemma allows us to identify the $(\rho_A, \rho_B)$-maps which undergo a sequence of Rauzy inductions corresponding to the word $w$:

\begin{proposition}
\label{prop:9}
    Let $M_\ell = \begin{pmatrix} a & b \\ c & d \end{pmatrix}$ be the length transformation matrix for the word $w$. The interval $I(w) \subseteq D$ of parameters $x_T$ whose corresponding $(\rho_A, \rho_B)$-map may undergo Rauzy inductions corresponding to the word $w$ is then
    \[
        I(w) = \left[ \frac{-b}{a - b}, \frac{d}{d - c} \right]
    \]
\end{proposition}
\begin{proof}
    Let $x_T \in D$ and let $T: D \to D$ be the $(\rho_A, \rho_B)$-map where $A = [0, x_T)$ and $B = (x_T, 1]$.
    \\
    $M_\ell$ sends the initial length vector $\lambda = (x_T, 1 - x_T)^\top$ to $\lambda' = ((a - b)x_T + b, (c - d)x_T + d)^\top$.
    \\
    $w$ is a valid sequence of Rauzy steps to apply to $T$ if and only if both elements of $\lambda'$ are nonnegative.  
    But this says $x_T \geq -b/(a - b)$ and $x_T \leq d/(d - c)$.  
    By Lemma \ref{lem:8}, these two inequalities define the desired interval of parameters $I(w)$.
\end{proof}

Within $I(w)$, there is a subset $H(w)$ consisting of discontinuity points $x_T$ such that Rauzy induction on the corresponding map terminates after the sequence of moves $w$.  
In fact, $H(w)$ is a subinterval of $I(w)$ which may be determined in terms of $M_\ell$:

\begin{proposition}
\label{prop:10}
    The interval $H(w) \subseteq I(w)$ parameterizing $(\rho_A, \rho_B)$-maps which reach the terminating case (2) after valid Rauzy inductions corresponding to the word $w$ is
    \[
        H(w) = \left[ 
            \frac{d - b\rho_A^{-m_A}\rho_B^{-m_B}}{(a - b)\rho_A^{-m_A}\rho_B^{-m_B} + d - c}, 
            \frac{d - b\rho_A^{n_A}\rho_B^{n_B}}{(a - b)\rho_A^{n_A}\rho_B^{n_B} + d - c} 
        \right]
    \]
\end{proposition}

\begin{proof}
    If $x_T \in H(w)$, then the final transformation after all inductions, $T'$, must be in case (2).  
    In terms of interval lengths, this means that 
    $\lambda'_A \geq \rho_A^{n_A}\rho_B^{n_B} \lambda'_B$ 
    and $\lambda'_B \geq \rho_A^{m_A}\rho_B^{m_B} \lambda'_A$.
    Using $\lambda' = M_\ell (x_T, 1-x_T)^\top$, these inequalities define the desired interval of parameters $H(w)$.
\end{proof}

Note that when $w \neq v$, $H(w) \cap H(v) = \emptyset$.  
This is because the Rauzy induction algorithm has no ambiguity, 
there is only one valid induction word of length $k$ for any given $(\rho_A, \rho_B)$-map.  
Let $H_k := \bigcup_{|w| \leq k} H(w)$ parameterize the set of all $(\rho_A, \rho_B)$-maps which terminate after at most $k$ Rauzy inductions, 
and $H = \bigcup_k H_k$ parameterize the set of all $(\rho_A, \rho_B)$-maps which terminate at all under the Rauzy algorithm.  
$H$ has the same construction as the complement of the Cantor triadic set.  
Indeed, $H_k$ adds $2^k$ intervals, one within each gap between the intervals in $H_{k-1}$.  
These gaps are the intervals $I(w)$ where $|w| = k$.  
For some $(\rho_A, \rho_B)$-map, $H_2$ might look like the set in Figure \ref{fig:H2}.

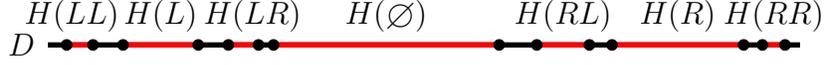
\begin{figure}[ht]
\begin{center}
    \begin{tikzpicture}
        \draw [line width = 0.07cm] (0, 0) -- (10, 0) node[pos = 0, left, color = black]{$D$};
        
        \draw [line width = 0.07cm, color = red] (3, 0) -- (6, 0) node[pos = 0.5, above, color = black]{$H(\emptyset)$};
        
        \draw [line width = 0.07cm, color = red] (1, 0) -- (2, 0) node[pos = 0.5, above, color = black]{$H(L)$};
        \draw [line width = 0.07cm, color = red] (7.5, 0) -- (9.25, 0) node[pos = 0.5, above, color = black]{$H(R)$};
        
        \draw [line width = 0.07cm, color = red] (0.25, 0) -- (0.6, 0) node[pos = 0.2, above, color = black]{$H(LL)$};
        \draw [line width = 0.07cm, color = red] (2.4, 0) -- (2.8, 0) node[pos = 0.8, above, color = black]{$H(LR)$};
        \draw [line width = 0.07cm, color = red] (6.5, 0) -- (7.2, 0) node[pos = 0.5, above, color = black]{$H(RL)$};
        \draw [line width = 0.07cm, color = red] (9.5, 0) -- (9.8, 0) node[pos = 0.5, above, color = black]{$H(RR)$};
        
        \draw[fill] (3, 0) circle (2pt);
        \draw[fill] (6, 0) circle (2pt);
        \draw[fill] (1, 0) circle (2pt);
        \draw[fill] (2, 0) circle (2pt);
        \draw[fill] (7.5, 0) circle (2pt);
        \draw[fill] (9.25, 0) circle (2pt);
        \draw[fill] (0.25, 0) circle (2pt);
        \draw[fill] (0.6, 0) circle (2pt);
        \draw[fill] (2.4, 0) circle (2pt);
        \draw[fill] (2.8, 0) circle (2pt);
        \draw[fill] (6.5, 0) circle (2pt);
        \draw[fill] (7.2, 0) circle (2pt);
        \draw[fill] (9.5, 0) circle (2pt);
        \draw[fill] (9.8, 0) circle (2pt);
    \end{tikzpicture}
\caption{An example of $H_2$ in red.  Each interval corresponds to a set $H(w)$ for $|w| \leq 2$.}
\label{fig:H2}
\end{center}
\end{figure}

For every $x_T \in H$, since the Rauzy induction terminates after finitely many steps, we can find a periodic orbit for $T$ beginning at the attractive fixed point of the final map $T'$.  
This orbit then attracts all other orbits and corresponds to an attracting orbit for the directional foliation on a corresponding dilation surface.  

We now check that $H$ is of full measure in $[0, 1]$.  
This is where we utilize $\rho_A, \rho_B < 1$.
The following intermediate bound does much of the heavy lifting:

\begin{lemma}
\label{lem:11}
    Let $\rho := \max\{\rho_A, \rho_B\}$ and let 
    $M_\ell = \begin{pmatrix} a & b \\ c & d \end{pmatrix}$.  
    The quantity $\displaystyle s := \frac{a - b}{d - c}$ satisfies
    \[
        1 - \rho < s < \frac{1}{1 - \rho}.
    \]
\end{lemma}

\begin{proof}
    We induct on word length.  
    When $\ell = 0$, $M_0 = I$, $s = 1$, so $1 - \rho < 1 < \frac{1}{1 - \rho}$.
    
    Suppose the lemma holds when the length transformation matrix is
    $M_\ell = \begin{pmatrix} a & b \\ c & d\end{pmatrix}$ with exponents $(m_A, m_B, n_A, n_B)$.  Recall that $M_{\ell + 1}$ is attained either by multiplying $M_\ell$ by the right-induction matrix $R_{m_A, m_B}$ or the left-induction matrix $L_{n_A, n_B}$.  We can bound the new value of $s$, which we denote as $s_R$ or $s_L$ respectively, as follows. 
    
    For $R_{m_A, m_B} M_\ell$, the quantity is $s_R = \rho_A^{m_A}\rho_B^{m_B}s + 1$.  Since $\rho \geq \rho_A^{m_A}\rho_B^{m_B}$,
    \[
        1 - \rho < \rho_A^{m_A}\rho_B^{m_B} s + 1 
        < \frac{\rho_A^{m_A}\rho_B^{m_B}}{1 - \rho} + 1 \leq \frac{1}{1 - \rho}.
    \]
    
    For $L_{n_A, n_B} M_\ell$, the quantity is $s_L = \frac{s}{\rho_A^{n_A}\rho_B^{n_B} + s}$.  Since $\rho \geq \rho_A^{n_A}\rho_B^{n_B}$,
    \[
        1 - \rho \leq \frac{1 - \rho}{\rho_A^{n_A}\rho_B^{n_B} + 1 - \rho} 
        < \frac{s}{\rho_A^{n_A}\rho_B^{n_B} + s} 
        < \frac{1}{1 - \rho}.
    \]
    So the bound $1 - \rho < s < \frac{1}{1 - \rho}$ is preserved under left or right inductions.
\end{proof}

With the help of this bound, we may bound $|H(w)| / |I(w)|$ below:

\begin{proposition}
\label{prop:12}
    There exists an $N$ and $\delta > 0$ such that for all words $w$ with $|w| > N$, we have $|H(w)|/|I(w)| \geq \delta$. As a consequence, $H = \bigcup_w H(w)$ has full measure. 
    
    Thus, for any $\rho_A < 1, \rho_B < 1$, there is a measure zero Cantor set $C \subset [0,1]$ of break points $x_T$ such that on the complement of $C$, the corresponding $(\rho_A, \rho_B)$ map accumulates onto a periodic or critical orbit.
\end{proposition}

\begin{proof}
    In Lemma \ref{lem:11}, we found that if $M_\ell = \begin{pmatrix} a & b \\ c & d \end{pmatrix}$ is the length transformation matrix and $\rho = \max\{\rho_A, \rho_B\}$, then $s := \frac{a - b}{d - c}$ is bounded as $1 - \rho < s < \frac{1}{1 - \rho}$.  Using Propositions \ref{prop:9} and \ref{prop:10}, after some algebraic manipulations we find that 
    \[
        \frac{|H(w)|}{|I(w)|} 
        = \frac{1}{s\rho_A^{n_A}\rho_B^{n_B} + 1} - \frac{1}{s\rho_A^{-m_A}\rho_B^{-m_B} + 1} 
        \geq \frac{1 - \rho}{\rho_A^{n_A}\rho_B^{n_B} + 1 - \rho} - \frac{1}{(1 - \rho)\rho_A^{-m_A}\rho_B^{-m_B} + 1}.
    \]
    We denote the last expression above by $\delta(m_A, m_B, n_A, n_B)$.  Let $N > 2\log(1 - \rho) / \log(\rho)$.  \\
    If our initial word $w$ has length $|w| > N$, then $m_A + m_B + n_A + n_B \geq |w| > N$ and another algebraic manipulation shows $\delta(m_A, m_B, n_A, n_B) > 0$.
    Moreover, $\delta(m_A, m_B, n_A, n_B)$ is increasing in each index. 
    Let $\delta := \min\{ \delta(m_A, m_B, n_A, n_B) : m_A + m_B + n_A + n_B = N + 1\} > 0$.  
    Then $|H(w)| / |I(w)| \geq \delta$ for all $|w| > N$.
    
    Recall the intermediate sets $H_N = \bigcup_{|w| \leq N} H(w)$.  
    Now that we have $|H(w)| \geq \delta |I(w)|$ for all $|w| > N$, in every step of the Cantor construction of $H$ after the $N$th step we remove at least a $\delta$ proportion of what remains.  
    Then $|D \setminus H| \leq (1 - \delta)^n|D\setminus H_N|$ for all $n$, corresponding to the $(n + N)th$ step.  
    As $n \to \infty$, $|D\setminus H| = 0$. This shows that $H$ has full measure in $D$.
    
    We have already seen that $C := D \setminus H$ arises topologically from the same process as the Cantor middle thirds construction, perhaps modulo the endpoints of intervals used in the construction process.  Then $C$ is a measure zero Cantor set in the set of parameters of $(\rho_A, \rho_B)$-maps on whose complement we have maps with an attracting periodic or critical orbit.
\end{proof}

  Now, we can bring these results back to the dilation torus $X$ to prove part of Theorem \ref{thm:1}.  
Recall that on $X$ we have yet to understand the dynamics of directional foliations in three intervals of directions introduced in Section \ref{sec:flowStructure}: $I_1$, $I_3$, and $I_5$. What we have proven thus far handles the dynamics in $I_3$.

\begin{proposition}
\label{prop:13}
    Let $X$ be a one-holed dilation torus.  
    Within the interval $I_3$ as defined in Section \ref{sec:flowStructure},
    there is a full measure set of directions in which the flow is attracted to a periodic orbit or a saddle connection.  
    The complement of this set in $I_3$ is a Cantor set.
\end{proposition}

\begin{proof}
    Directions in $I_3$ point toward the longer sides of the polygonal model, and away from the shorter sides. 
    Then as we saw in Section \ref{sec:flowStructure}, the first return map is a $(\rho_A, \rho_B)$-map with $\rho_A < 1$ and $\rho_B < 1$.  
    Moreover, as the direction sweeps across $I_3$, in the first return map $\rho_A$ and $\rho_B$ are constant while the discontinuity point $x_T$ sweeps smoothly across $[0, 1]$.  
    This provides a smooth bijection between $I_3$ and $x_T \in [0, 1]$, which in particular preserves the measure zero Cantor set $C$.
    
    Then in $I_3$, the set of directions whose first return map has an attractive periodic orbit is full measure.  
    These attractive periodic orbits suspend to attractive periodic orbits on $X$.  
    Finally, the remaining directions in $I_3$ a measure zero Cantor set.
\end{proof}

\subsection{The expanding case}
\label{sec:pA>1}

If one of the intervals expands, the methods above need to be adjusted.  
Without loss of generality, we may assume $\rho_A \geq 1$ and $\rho_B < 1$.  

First, there is the concern of parameterizing $\mathcal{I}(\rho_A, \rho_B)$ by the discontinuity point $x_T$. 
A $(\rho_A, \rho_B)$-map with discontinuity point $x_T$ is injective if and only if $\rho_A x_T + \rho_B ( 1 - x_T) < 1$, i.e. $x_T \in \left[0, \frac{1 - \rho_B}{\rho_A - \rho_B}\right)$.  
We can check that Rauzy induction is done on the left when $x_T < \frac{\rho_B}{1 + \rho_B}$, Rauzy induction is done on the right when $x_T > \frac{1}{1 + \rho_A}$, 
and otherwise Rauzy induction terminates.  

So as long as $\rho_A\rho_B \geq 1$, we find that 
$x_T < \frac{1 - \rho_B}{\rho_A - \rho_B} \leq \frac{\rho_B}{1 + \rho_B}$.  
In this case, only left induction may occur.  
If instead $\rho_A\rho_B < 1$, we find that 
$\frac{1 - \rho_B}{\rho_A - \rho_B} > \frac{1}{1 + \rho_A}$, so all three possible cases of Rauzy induction may occur.  
In particular, right induction occurs for 
$x_T \in \left(\frac{1}{1 + \rho_A}, \frac{1 - \rho_B}{\rho_A - \rho_B}\right)$.
We note that this interval makes up a proportion of at most $\frac{\rho_A}{1 + \rho_A}$ of the original interval.

With this setup, we define a modified Rauzy algorithm for the case where $\rho_A \geq 1$.
We begin with a map $T_0 \in \mathcal{I}(\rho_A, \rho_B)$ and let $\rho_A(0) = \rho_A$, $\rho_B(0) = \rho_B$.  
We then produce maps $T_n \in \mathcal{I}(\rho_A(n), \rho_B(n))$ inductively by the following algorithm.  As we define the algorithm, we inductively show $\rho_B(n) < 1$ and $\rho_A(n + 1) \leq \rho_A(n)$ for each $n$ until the algorithm terminates.
\begin{enumerate}
    \item If $\rho_A(n)\rho_B(n) \geq 1$, then the only valid step is left induction. We perform left induction on $T_n$ to obtain a map $T_{n+1}$ with
    $\rho_A(n + 1) = \rho_A(n)\rho_B(n)$ and $\rho_B(n + 1) = \rho_B(n)$.  
    
    We repeat the algorithm on $T_{n + 1}$.
    
    We record that if $\rho_B(n) < 1$ then $\rho_A(n + 1) \leq \rho_A(n)$ and $\rho_B(n + 1) < 1$, so our next step is in case 2.
    \item If $\rho_A(n)\rho_B(n) < 1$, we have 3 possibilities:
    \begin{enumerate}
        \item If the valid step is left induction, we find $T_{n + 1}$ with 
        $\rho_A(n + 1) = \rho_A(n)\rho_B(n) < 1$ and $\rho_B(n + 1) = \rho_B(n) < 1$.  
        Then we are in the case $\mathcal{I}(\rho_A, \rho_B)$ with $\rho_A, \rho_B < 1$.  
        We already understand the behavior of such maps, so we are done.
        \item If the valid step is termination, we are done.
        \item If the valid step is right induction, we perform right induction on $T_n$ to obtain a map $T_{n + 1}$ with $\rho_A(n + 1) = \rho_A(n)$ and $\rho_B(n + 1) = \rho_A(n)\rho_B(n)$.  
        
        We repeat the algorithm on $T_{n + 1}$.
        
        We record that $\rho_A(n + 1) \leq \rho_A(n)$ and $\rho_B(n + 1) < 1$.
    \end{enumerate}
\end{enumerate}

\begin{lemma}
\label{lem:case1}
Given that $T$ is a $(\rho_A, \rho_B)$ map with $\rho_A \geq 1$ and $\rho_B < 1$, every instance of a case 1 step in the modified Rauzy algorithm must eventually be followed by a case 2 step. 
\end{lemma}
\begin{proof}
We suppose the $n$th step of modified Rauzy induction lands in case 1. Then, for the $k$th subsequent consecutive step which lands in case 1, 
\[\rho_A(n+k) = \rho_A(n + k - 1)\rho_B(n + k - 1) = \cdots = \rho_A(n)\rho_B(n)^k.\]
Since $\rho_B(n) < 1$, if every step after the $n$th lands in case 1, then eventually $\rho(n + k) < 1$ which contradicts the condition to be in case 1. Hence, there must exist some $(n+k)$th step for $k \geq 1$ that lands in case 2. 
\end{proof}

\begin{proposition}
\label{prop:expanding}
If $\rho_A \geq 1$ and $\rho_B < 1$ or $\rho_A < 1$ and $\rho_B \geq 1$, then there exists a measure zero Cantor set $C \subset [0,1]$ of break point locations $x_T \in [0,1]$ such that every $(\rho_A, \rho_B)$ map with break point in $[0,1] \backslash C$ has flow attracted to a periodic or critical orbit. 
\end{proposition}

\begin{proof}
We first consider the $\rho_A \geq 1$ and $\rho_B < 1$ case. By symmetry, the $\rho_A < 1$ and $\rho_B \geq 1$ case will follow similarly.

We wish to find a measure zero Cantor set $C$ such that every $(\rho_A, \rho_B)$-map with break point $x_T \in [0,1]\backslash C$ has an attracting periodic or critical orbit. 

We let the sequence of maps under the modified Rauzy induction algorithm be $T_n \in \mathcal{I}(\rho_A(n), \rho_B(n))$ and we will adopt the convention that we do not rescale the intervals on which $T_n$ are defined during induction. So each $T_n : I_n \rightarrow I_n$ and each $I_n \subset I_{n-1}$. We notice then that the location of the break point $x_T(n)$ of the map $T_n$ is an affine function of $x_T(n-1)$ that depends on what case the previous step was in. 

The algorithm never terminates in case 1 and by Lemma \ref{lem:case1}, every case 1 step is eventually followed by a case 2 step. We thus can restrict our attention to the case 2 steps. 

We let $n_1 \geq 1$ be the first time that we perform a case 2 Rauzy induction. That is, $T_{n_1 - 1}$ has that $\rho_A(n_1-1)\rho_B(n_1 -1) < 1$. Then, the possible $x_T(n_1-1)$ values are in a bijective affine coorespondence with the $x_T$ values of our original map $T_0$. Upon rescaling intervals to have length $1$, the $x_T(n_1-1)$ and $x_T$ are in a smooth bijective correspondence. In particular, Cantor sets of $x_T$ values correspond to Cantor sets of $x_T(n_1-1)$ values and vice versa. 

We have already observed that the interval of possible $x_T(n_1-1)$ values in a case 2 step breaks up into three intervals corresponding to case 2(a), 2(b), and 2(c) in that order from left to right. On the 2(a) interval, by Proposition \ref{prop:12}, we have a measure zero Cantor set of parameters on whose complement the map $T_0$ has an attracting periodic or critical orbit since $\rho_A(n_1) < 1$ and $\rho_B(n_1) < 1$. $T_0$ has an attracting periodic or critical orbit on all of 2(b). On 2(c), we must continue the Rauzy algorithm. Thus, our original interval $[0,1]$ of $x_T$ values also breaks up into three intervals $I_1^1, I_2^1, I_3^1$ from left to right with the same behaviors. We let $C_1 \subset I_1^1$ be the measure zero Cantor set in the first interval given by Proposition \ref{prop:12}.

We can then repeat this process, focusing now on the interval $I_3^1 \subset [0,1]$ and note that if step $n_2$ is the next time that we are in case 2, then at that step $I_3^1$ will further break into three intervals $I_1^2, I_2^2, I_3^2$. We again have a Cantor set $C_2 \subset I_1^2$ such that on the $I_1^2 \backslash C_2$ and on $I_2^2$, $T_0$ has an attracting periodic or critical orbit. We repeat this process inductively with our third intervals $I_3^k$. 

We recall our observation that if the $n$th step is in case 2, then the proportion of the set of parameters that lands in case 2(c) is at most $\frac{\rho_A(n)}{1 + \rho_A(n)}$. Each $\rho_A(n) \leq \rho_A$, so the set $S$ of parameters $x_T \in [0,1]$ that reach case 2(c) infinitely often has size at most 
\[
    |S| \leq \prod_{j = 1}^\infty \frac{\rho_A(n_j)}{1 + \rho_A(n_j)} \leq \prod_{j = 1}^\infty \left( \frac{\rho_A}{1 + \rho_A} \right) = 0.
\]
Thus, this set $S$ of parameters $x_T \in [0,1]$ is just the right endpoint $x_T = 1$. We let $C = \left(\bigcup_{k=1}^\infty C_k\right) \cup \{1\}.$ This is the countable union of measure zero Cantor sets and one point that is a limit point of these Cantor sets. Thus, $C$ is also a Cantor set.

We have seen that every $(\rho_A, \rho_B)$-map with break point in $[0,1] \backslash C$ ends in a terminating Rauzy induction step after finitely finitely steps and so has flow attracted to a periodic or critical orbit. 
\end{proof}

We now have all of the tools that we need to prove the main theorem of this section.

\begin{proof}[Proof of Theorem \ref{thm:1}]
By Proposition \ref{prop:13}, there is a Cantor set $C_3$ in the interval $I_3$ of directions introduced in Section \ref{sec:flowStructure} off of which our $(\rho_A, \rho_B)$-map has an attracting periodic or critical orbit. This corresponds to an attracting periodic cycle or saddle connection for the directional flow on our one-holed dilation torus. On the intervals $I_1$ and $I_5$, Proposition \ref{prop:expanding} gives us Cantor sets $C_1 \subset I_1$ and $C_5 \subset I_5$ off of which the corresponding $(\rho_A, \rho_B)$-maps again have  an attracting periodic or critical orbit. This implies that the straight-line flow on $X$ in these directions accumulates on a periodic orbit or a saddle connection. 

We also saw in Section \ref{sec:I2directions} that the straight line flow in the whole interval of directions $I_2$ and $I_4$ has an attracting periodic orbit or saddle connection. Thus, letting $C = C_1 \cup C_3 \cup C_5$, which is a union of three disjoint Cantor sets and is thus a Cantor set, we see that in every direction of $\mathbb{RP}^1 \backslash C$, the straight line flow on $X$ accumulates to a periodic orbit or saddle connection. 
\end{proof}

\section{Dynamics of maps in $C$}
\label{sec:cantorMaps}

We have determined a Cantor set of directions outside of which the dynamics are already understood.  In this section, the goal is to prove the following theorem: 

\dynamics

The $m_B$ direction was handled in Section \ref{sec:boundaryParallelDirection}. For the other directions, it suffices to prove an analogous statement for the first return map to a transversal. This reduces to understanding the dynamics of $(\rho_A, \rho_B)$-maps. 

We will continue with the assumption that $\rho_A \geq \rho_B$ and that $\rho_B < 1$.  
Let $T : [0, 1] \to [0,1]$ be a $(\rho_A, \rho_B)$-map whose discontinuity point $x_T$ is in this Cantor set of parameters $C$ where we did not yet find an attractive periodic orbit. Then the Rauzy algorithm does not terminate on $T$. Except in countably many cases, we claim that there is a measure zero Cantor set  in $[0, 1]$ which is the accumulation set of every orbit of $T$.

We will prove this claim with the following steps:
\begin{enumerate}
    \item We handle the dynamics in the countably many exception cases. This is the set $C_1$ of Theorem \ref{thm:2}
    \item We define a sequence of maps $f_n$ sending the discontinuity point $x_T$ to the measures of sets in $[0,1]$ which are disjoint from accumulation sets of orbits of $T$. 
    \item We show that $f_n$ are continuous and nonincreasing, and that for $x_T$ in a dense set, $f_n(x_T) \xrightarrow{n\to\infty} 1$.  This will prove that $f_n(x_T) \xrightarrow{n\to\infty} 1$ for all $x_T$.  Then we've shown that for any $(\rho_A, \rho_B)$-map $T$, the $\omega$-limit set $\omega(x)$ is measure zero for any $x \in [0,1]$.  
    \item We show that for directions in $C_2$, the limit set $\omega(x)$ is in fact a Cantor set that is independent of $x$.
\end{enumerate}

\subsection{Dynamics in the $C_1$ directions} 

Let us begin with the countably many exception cases in Step 1, the set of direction $C_1$. These will correspond to $(\rho_A, \rho_B)$-maps for which the Rauzy algorithm does not terminate but only performs finitely many left or finitely many right Rauzy induction steps. 

\begin{proposition}
\label{prop:C1}
If the Rauzy induction algorithm on a $(\rho_A, \rho_B)$-map $T: [0,1] \rightarrow [0,1]$ does not terminate but only performs finitely many left or finitely many right Rauzy induction steps, 
then $T$ is a degenerate $(\rho_A, \rho_B)$-map such that every orbit is attracted to a terminating orbit 
(the orbit of a point that eventually hits an endpoint of $A$ or $B$).  
\end{proposition}

\begin{proof}
Let us consider the case of finitely many right Rauzy inductions. 
Then, after finitely many steps of Rauzy induction, we have some $(\rho_A, \rho_B)$-map $T$ such that all future Rauzy induction steps are left inductions. 
After rescaling, we may assume that $T: [0,1] \rightarrow [0,1]$ and the interval is divided into intervals $A$ and $B$ of lengths $\lambda_A$ and $\lambda_B = 1 - \lambda_A$ at the point $x_T \in [0,1]$.

\begin{figure}[ht]
\begin{center}
\begin{tikzpicture}[scale = 1.5]
  \coordinate (A) at (0, 1.5);
  \coordinate (B) at (2, 1.5);
  \coordinate (C) at (7, 1.5);
  \coordinate (D) at (0, 0);
  \coordinate (E) at (6, 0);
  \coordinate (F) at (6.5, 0);
  \coordinate (G) at (7, 0);
  \coordinate (H) at (3.5, 1.5);
  \coordinate (I) at (4.3, 1.5);

    \draw[dashed] (2,2.5) -- (2,-1); 
    \draw[dashed] (3.5,2.5) -- (3.5,-1); 
    \draw[dashed] (4.3,2.5) -- (4.3,-1); 
    
    \node at (2.75, 2.25) {$A'$}; 
    \node at (5.25, 2.25) {$B'$};
    \node at (3.9,-.75) {$T'(B')$};
    \node at (6.75,-.75) {$T'(A')$};
    \draw[very thick, color = blue] (2,2) -- (3.5, 2); 
    \draw[very thick, color = red] (3.5,2) -- (7, 2);
    
    \filldraw [color = black] (3.5,2) circle (2pt);
 
    \draw[very thick, color = red] (2,-.5) -- (6, -.5); 
    \draw[very thick, color = blue] (6.5,-.5) -- (7, -.5);   
    
    \node at (3.9, 2.75) {$A''$}; 
    \node at (5.65, 2.75) {$B''$};
    \node at (4.75,-1.25) {$T''(B'')$};
    \node at (6.75,-1.25) {$T''(A'')$};
    
    \draw[very thick, color = blue] (3.5,2.5) -- (4.3, 2.5); 
    \draw[very thick, color = red] (4.3,2.5) -- (7, 2.5);
    
    \filldraw [color = black] (4.3,2.5) circle (2pt);
    
    \draw[very thick, color = red] (3.5,-1) -- (6, -1); 
    \draw[very thick, color = blue] (6.5,-1) -- (7, -1);

  \draw [very thick, color = blue] (A) -- (B) 
    node[pos = 0.5, above, color = black]{$A$} 
    node[pos = 0.5, below, color = black]{$\lambda_A$}; 
  \draw [very thick, color = red] (B) -- (C)
    node[pos = 0.5, above, color = black]{$B$};
    \draw [very thick, color = red] (B) -- (H) node[pos = 0.5, below, color = black]{$\rho_B^{-1} \lambda_A$};
     \draw [very thick, color = red] (H) -- (I) node[pos = 0.5, below, color = black]{$\rho_B^{-2} \lambda_A$};
  \draw [very thick, color = blue] (F) -- (G) 
    node[pos = 0.5, below, color = black]{$T(A)$};
  \draw [very thick, color = red] (D) -- (E)
    node[pos = 0.5, below, color = black]{$T(B)$}
    node[pos = 0.5, above, color = black]{$\rho_B \lambda_B$};
  
  \filldraw [color = black] (B) circle (2pt) node[above]{$x_T$};
  
\end{tikzpicture}
\caption{A series of left inductions on a $(\rho_A, \rho_B)$-map with discontinuity point $x_T$.}
\label{fig:leftinduction}
\end{center}
\end{figure}
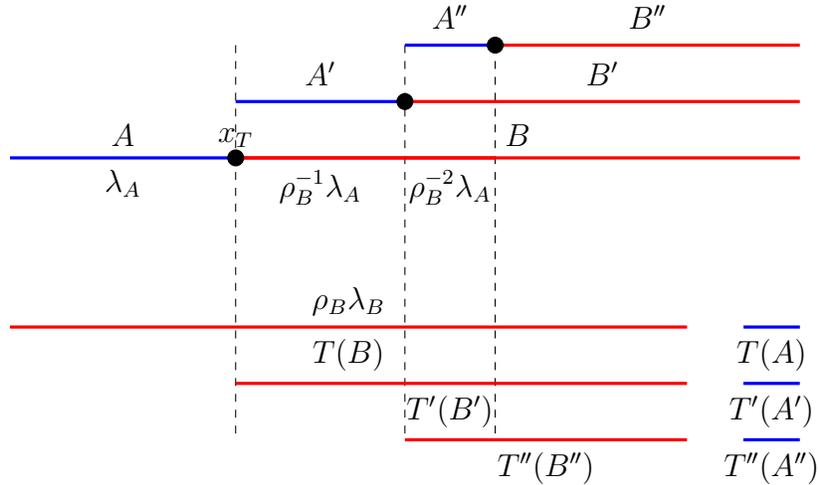

Referring to Figure \ref{fig:leftinduction}, we see that to have an infinite series of left inductions, we always need the breakpoint $x_T$ after each iteration of Rauzy induction to be to the left of the right endpoint of the image of the $B$ interval. 
In our figure, we let $T$ be the initial $(\rho_A, \rho_B)$ with intervals $A$ and $B$. 
Then $T'$ is the map after one iteration of Rauzy induction with intervals $A'$ and $B'$, and $T''$ with intervals $A''$ and $B''$ is the map after two iterations of Rauzy induction. 

We see then that if we do not rescale the maps, the right endpoint of the $B$ (or $B'$, or $B''$, etc.) interval is always at the point $\rho_B\lambda_B$. 
So our condition for infinitely many left inductions is that $\lambda_A + \rho_B^{-1}\lambda_A + \rho_B^{-2} \lambda_A + \ldots \leq \rho_B\lambda_B.$ 
Rewriting, we have that $$\lambda_A (1+ \rho_B^{-1}+ \rho_B^{-2}  + \ldots) \leq \rho_B\lambda_B.$$

The right hand side $\rho_B \lambda_B \leq 1$. When $\rho_B \leq 1$ the left hand side is infinite unless $\lambda_A = 0$. 
When $\rho_B > 1$, we can rewrite the inequality at $\frac{\lambda_A}{1- \rho_B^{-1}} \leq \rho_B \lambda_B$ and rearranging with $\lambda_A = (1-\lambda_B)$ gives us that $1 \leq \rho_B\lambda_B$. 
In this case, we must have that $\rho_B \lambda_B = 1$. 
Pictures of these two cases are show in \ref{fig:infiniteleft}. 

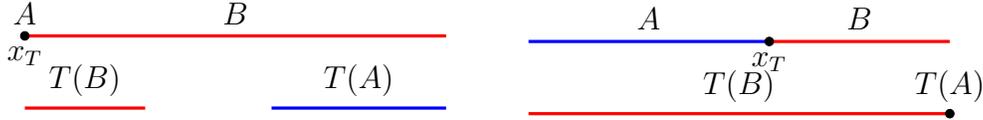
\begin{figure}[ht]
\begin{center}
\begin{tikzpicture}[scale=0.8]
  \coordinate (A) at (0, 1.2);
  \coordinate (B) at (0, 1.2);
  \coordinate (C) at (7, 1.2);
  \coordinate (D) at (0, 0);
  \coordinate (E) at (2, 0);
  \coordinate (F) at (4.1, 0);
  \coordinate (G) at (7, 0);

  \draw [very thick, color = blue] (A) -- (B) 
    node[pos = 0.5, above, color = black]{$A$};
  \draw [very thick, color = red] (B) -- (C) 
    node[pos = 0.5, above, color = black]{$B$};
  \draw [very thick, color = blue] (F) -- (G) 
    node[pos = 0.5, above, color = black]{$T(A)$};
  \draw [very thick, color = red] (D) -- (E) 
    node[pos = 0.5, above, color = black]{$T(B)$};
  
  \filldraw [color = black] (B) circle (2pt) node[below]{$x_T$};
  
  \node at (8,0) {}; 
\end{tikzpicture}
\begin{tikzpicture}[scale=0.8]
  \coordinate (A) at (0, 1.2);
  \coordinate (B) at (4, 1.2);
  \coordinate (C) at (7, 1.2);
  \coordinate (D) at (0, 0);
  \coordinate (E) at (7, 0);
  \coordinate (F) at (7, 0);
  \coordinate (G) at (7, 0);

  \draw [very thick, color = blue] (A) -- (B) 
    node[pos = 0.5, above, color = black]{$A$};
  \draw [very thick, color = red] (B) -- (C) 
    node[pos = 0.5, above, color = black]{$B$};
  \draw [very thick, color = blue] (F) -- (G) 
    node[pos = 0.5, above, color = black]{$T(A)$};
  \draw [very thick, color = red] (D) -- (E) 
    node[pos = 0.5, above, color = black]{$T(B)$};
  
  \filldraw [color = black] (B) circle (2pt) node[below]{$x_T$};
  \filldraw [color = black] (F) circle (2pt);
\end{tikzpicture}
\caption{The two cases of $(\rho_A, \rho_B)$-map with an infinite sequence of only left inductions.}
\label{fig:infiniteleft}
\end{center}
\end{figure}

We can see that each of these is a degenerate $(\rho_A, \rho_B)$-map where some interval has collapsed to the zero interval. If we were to analyze the dynamics of these maps anyways, we see that in both cases every orbit is attracted to the left or right endpoint of the interval. The case of only right inductions can be similarly concluded. 

If we had started with a $(\rho_A, \rho_B)$-map that eventually ends in an infinitely string of left or right inductions, then every orbit following the Rauzy induction algorithm backward from the point of the only left or only right cases shows that every orbit of our original map converges on a terminating orbit. 
\end{proof}

We make a note here that how we should think about the case of the Rauzy algorithm ending in a string of infinite left or infinite right inductions is that the algorithm really should have terminated in the previous step with $x_T$ equaling one of the endpoints of $T(A)$ or $T(B)$. This is similar to the phenomenon of ternary expansions of endpoints of the middle thirds Cantor set having both an infinte and finite representation (e.g. $0.0\overline{2}_3 = 0.1_3$).

\subsection{Dynamics in the $C_2$ directions}

We now need to deal with the situation in which $T$ undergoes infinitely many left and right inductions, the directions falling in $C_2$.  We want to identify points which cannot possibly be limit points of any orbit in order to define the functions described in Step 2 of the proof outline.  No point in the interval $I_\infty := D \setminus T(D) = (T(1), T(0))$, the complement of the image of $T$, is a limit point of any orbit.  See Figure \ref{fig:Iinfty}.

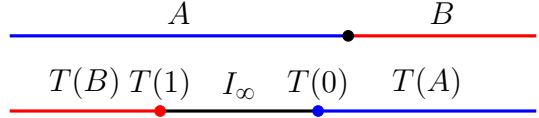
\begin{figure}[ht]
\begin{center}
\begin{tikzpicture}
    \draw [very thick, color = blue] (0, 1) -- (4.5, 1) node[pos = 0.5, above, color = black]{$A$};
    \draw [very thick, color = red] (4.5, 1) -- (7, 1) node[pos = 0.5, above, color = black]{$B$};
    \draw [very thick, color = blue] (4.1, 0) -- (7, 0) node[pos = 0.5, above, color = black]{$T(A)$};
    \draw [very thick, color = red] (0, 0) -- (2, 0) node[pos = 0.5, above, color = black]{$T(B)$};
    \draw [very thick] (2, 0) -- (4.1, 0) node[pos = 0.5, above, color = black]{$I_\infty$};
    
    \draw[fill] (4.5, 1) circle (2pt);
    \filldraw[color = blue] (4.1, 0) circle (2pt) node[above, black]{$T(0)$};
    \filldraw[color = red] (2, 0) circle (2pt) node[above, black]{$T(1)$};
\end{tikzpicture}
\caption{The interval $I_\infty = D \setminus T(D)$}
\label{fig:Iinfty}
\end{center}
\end{figure}

Similarly, for any $n$, no point in the interior of $T^n(I_\infty)$ is a limit point of any orbit.  
Indeed, for any $x$, $T^j(x)$ will not fall in $T^n(I_\infty)$ whenever $j > n$ because $T^{n-j}(I_\infty)$ is empty.  
Then $\bigcup_{k = 1}^\infty T^k(I_\infty)$ is disjoint from $\omega(x)$ for any $x$.  
For similar reasons, $T^m(I_\infty) \cap T^n(I_\infty) = \emptyset$ whenever $m \neq n$.  
We claim the disjoint union $\bigcup_{k}T^k(I_\infty)$ is full measure.

Recall that the discontinuity point $x_T$ parameterizes the set of $(\rho_A, \rho_B)$-maps, 
where $x_T$ lies in some interval $J = \left[0, \min\left\{1, \frac{1 - \rho_B}{\rho_A - \rho_B}\right\}\right]$ (if $\rho_A = \rho_B$, just let $J = [0, 1]$).  
Consider functions $f_n : J \to [0, 1]$ which map $x_T$ to the quantity $\left|\bigcup_{k = 0}^n T^k(I_\infty)\right| = \sum_{k = 0}^n |T^k(I_\infty)|$ (where $T$ is the map corresponding to $x_T$).  
Here is the rather technical but important proposition.

\begin{proposition}
\label{prop:nonincreasing}
If $\rho_A \geq \rho_B$, then for each $n$, the function $f_n : J \to [0, 1]$ mapping $x_T$ to $\left|\bigcup_{k = 0}^n T^k(I_\infty)\right|$ is continuous and nonincreasing.
\end{proposition}

Before providing the proof, it is informative to see the plot of such a function in Figure \ref{fig:fplot}.  The function is piecewise linear, nonincreasing, and the pieces alternate between having steep and shallow slopes.  Many pieces have slope 0.  These observations inform our proof.

\begin{figure}[ht]
\centering\includegraphics[scale = 0.8]{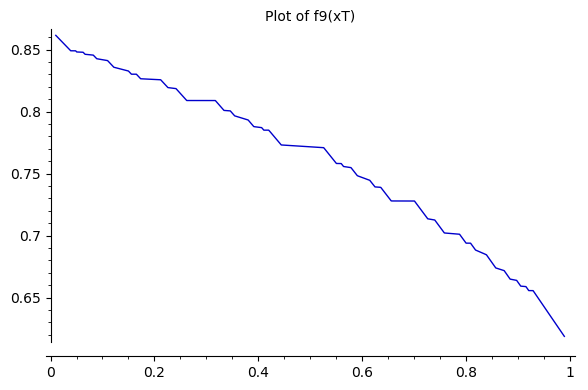}
\caption{Plot of $f_9(x_T)$ where $\rho_A = 0.9$, $\rho_B = 0.8$.  Produced in SageMath.}
\label{fig:fplot}
\end{figure}

\begin{proof}[Proof of Proposition \ref{prop:nonincreasing}]
We wish to show that $f_n : x_T \rightarrow |\bigcup_{k=0}^n T^k(I_\infty)|$ is continuous and nonincreasing. Let us first focus on points $x_T$ that are in the interior of some interval $T^j(I_\infty)$ for $0 \leq j \leq n$. The set of $x_T$ satisfying this condition is open. If $x_T \in T^j(I_\infty)$, then $I_\infty$ splits into two intervals $I_\infty^1 = (T(1), T^{-j}(x_T))$ and $I_\infty^2 = (T^{-j}(x_T), T(0))$.  
Let $w_T(y)$ be the sequence of intervals where $y \in D$ lands under $T$.
$w_T(y) =: w_T^1$ is the same for all $y \in I_\infty^1$, 
and $w_T(y) =: w_T^2$ is the same for all $y \in I_\infty^2$.  
Moreover, for $x_T'$ sufficiently close to $x_T$, the words $w_T^1$ and $w_T^2$ do not change. This is because the positions of $I_\infty^1$ and $I_\infty^2$ and their images vary continuously as $x_T$ changes, so they do not cross between $A$ and $B$ for small enough perturbations.  This shows continuity at $x_T$.

We will show that at such points $x_T$, $f_n' \leq 0$. To do so, we first make the following key observation: With $j$, $I_\infty^1$ and $I_\infty^2$ as before, we note that $T^{j+1}(I_\infty)$ consists of two intervals, one with left endpoint $0$ and the other with right endpoint $1$.  
Then $T^{j+2}(I_\infty)$ reattaches to $I_\infty$, with $T^{j+2}(I_\infty^2)$ attaching on the left of $I_\infty$ and $T^{j+2}(I_\infty^1)$ attaching on the right of $I_\infty$. 
Let $p = j + 2$, the period of this pattern.  
Further images $T^k(I_\infty^2)$ will always attach to the right of $T^{k-p}(I_\infty^2)$ and images $T^k(I_\infty^1)$ will always attach to the left of $T^{k-p}(I_\infty^1)$. 

We can see this process in Figure \ref{fig:prop14nonincreasing}, where $j=1$ and $p = 3$. 

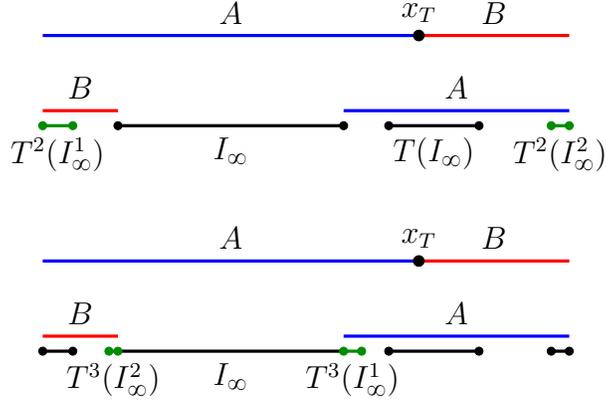
\begin{figure}[ht]
\begin{center}
\begin{tikzpicture}
    \pgftransformyshift{3cm}
    \draw [very thick, color = blue] (0, 1) -- (5, 1) node[pos = 0.5, above, color = black]{$A$};
    
    \draw node at (5,1.3) {$x_T$}; 
    
    \draw [very thick, color = red] (5, 1) -- (7, 1) node[pos = 0.5, above, color = black]{$B$};
    \draw [very thick, color = blue] (4, 0) -- (7, 0) node[pos = 0.5, above, color = black]{$A$};
    \draw [very thick, color = red] (0, 0) -- (1, 0) node[pos = 0.5, above, color = black]{$B$};
    \draw [very thick] (1, -0.2) -- (4, -0.2) node[pos = 0.5, below, color = black]{$I_\infty$};
    \draw [very thick] (4.6, -0.2) -- (5.8, -0.2) node[pos = 0.5, below, color = black]{$T(I_\infty)$};
    \draw [very thick, black!45!green] (0, -0.2) -- (0.4, -0.2) node[pos = 0.5, below, color = black]{$T^2(I^1_\infty)$};
    \draw [very thick, black!45!green] (6.76, -0.2) -- (7, -0.2) node[pos = 0.5, below, color = black]{$T^2(I^2_\infty)$};
    
    \filldraw (5, 1) circle (2pt);
    \filldraw (4, -0.2) circle (1.5pt);
    \filldraw (1, -0.2) circle (1.5pt);
    \filldraw (4.6, -0.2) circle (1.5pt);
    \filldraw (5.8, -0.2) circle (1.5pt);
    \filldraw[black!45!green] (0, -0.2) circle (1.5pt);
    \filldraw[black!45!green] (7, -0.2) circle (1.5pt);
    \filldraw[black!45!green] (0.4, -0.2) circle (1.5pt);
    \filldraw[black!45!green] (6.76, -0.2) circle (1.5pt);

    \pgftransformyshift{-3cm}
    \draw [very thick, color = blue] (0, 1) -- (5, 1) node[pos = 0.5, above, color = black]{$A$};
    
    \draw node at (5,1.3) {$x_T$}; 
    
    \draw [very thick, color = red] (5, 1) -- (7, 1) node[pos = 0.5, above, color = black]{$B$};
    \draw [very thick, color = blue] (4, 0) -- (7, 0) node[pos = 0.5, above, color = black]{$A$};
    \draw [very thick, color = red] (0, 0) -- (1, 0) node[pos = 0.5, above, color = black]{$B$};
    \draw [very thick] (1, -0.2) -- (4, -0.2) node[pos = 0.5, below, color = black]{$I_\infty$};
    \draw [very thick] (4.6, -0.2) -- (5.8, -0.2);
    \draw [very thick] (0, -0.2) -- (0.4, -0.2);
    \draw [very thick] (6.76, -0.2) -- (7, -0.2);
    \draw [very thick, black!45!green] (4, -0.2) -- (4.24, -0.2) node[pos = 0.5, below, color = black]{$T^3(I^1_\infty)$};
    \draw [very thick, black!45!green] (0.88, -0.2) -- (1, -0.2) node[pos = 0.5, below, color = black]{$T^3(I^2_\infty)$};
    
    \filldraw (5, 1) circle (2pt);
    \filldraw[black!45!green] (4, -0.2) circle (1.5pt);
    \filldraw[black!45!green] (1, -0.2) circle (1.5pt);
    \filldraw (4.6, -0.2) circle (1.5pt);
    \filldraw (5.8, -0.2) circle (1.5pt);
    \filldraw (0, -0.2) circle (1.5pt);
    \filldraw (7, -0.2) circle (1.5pt);
    \filldraw (0.4, -0.2) circle (1.5pt);
    \filldraw (6.76, -0.2) circle (1.5pt);
    \filldraw[black!45!green] (0.88, -0.2) circle (1.5pt);
    \filldraw[black!45!green] (4.24, -0.2) circle (1.5pt);
\end{tikzpicture}
\caption{Images of $I^1_\infty$ and $I^2_\infty$ split up under $T^2$, then reattach to $I_\infty$ under $T^3$.}
\label{fig:prop14nonincreasing}
\end{center}
\end{figure}

Thus, $\bigcup_{k=0}^n T^k(I_\infty)$ for $k \geq p$ consists of exactly $p+1$ disjoint intervals. Since the endpoints of $I_\infty$ are $T(0)$ and $T(1)$, we can see that the right endpoints of these $p+1$ intervals are $1, T^{n+1}(0), T^n(0), \ldots, T^{n-j}(0)$, and the left endpoints of these $p+1$ intervals are $0, T^{n+1}(0), T^n(0), \ldots, T^{n-j}(0)$. 

It follows then that $$f_n'(x_T) = \sum_{k=n-j}^{n+1} \left(\frac{d}{dx_T} (T^k(0)) - \frac{d}{dx_T}(T^k(1))\right).$$

We now make a few observations. First, we have that $T(0) = 1 - \rho_Ax_T$ and $T(1) = \rho_B(1-x_T)$. We also notice that $T^{k+1}(1) = aT^k(1) + b$, where $a = \rho_A$ or $\rho_B$ depending whether $T^k(1)$ is in the interval $A$ or $B$, and a similar statement holds for $T^{k+1}(0)$. Thus, 
\begin{align*} 
\frac{d}{dx_T} (T^k(0)) & = -\rho_A \cdot \rho_1^0 \rho_2^0 \cdots \rho_{k-1}^0 \\
\frac{d}{dx_T} (T^k(1)) & = -\rho_B \cdot \rho_1^1 \rho_2^1 \cdots \rho_{k-1}^1 
\end{align*} 

where $\rho_i^\ell$ is $\rho_A$ if $T^i(\ell) \in A$ and $\rho_B$ if $T^i(\ell) \in B$.  

We also note that $\rho_i^1$ and $\rho_i^0$ both follow a cyclic pattern of period $p$ where they are the equal except at two times in the cycle where they are opposite from one another. Specifically, for all integral $m \geq 1$, we have that
$$\begin{cases} \rho_i^0 = \rho_B, \rho_i^1 = \rho_A, & i = mp - 1 \\ \rho_i^0 = \rho_A, \rho_i^1 = \rho_B, & i = mp \\ \rho_i^0 = \rho_i^1, & \text{otherwise.} \end{cases}$$
Putting this together, we see that $\frac{d}{dx_T} (T^k(0)) - \frac{d}{dx_T}(T^k(1)) = 0$ for $k = mp-1$ and $\frac{d}{dx_T} (T^k(0)) - \frac{d}{dx_T}(T^k(1)) = (\rho_B-\rho_A) \cdot \overline{\rho}_k$ where $\overline{\rho}_k$ is a product of $\rho_A$ and $\rho_B$ terms when $k \neq mp-1$. In the latter case, $\frac{d}{dx_T} (T^k(0)) - \frac{d}{dx_T}(T^k(1)) < 0$ since $\rho_A > \rho_B$ and all $\overline{\rho}_k > 0$. Thus, $f_n'(x_T)$ is the sum of non-positive terms and is therefore non-positive. It follows that $f_n$ is continuous and non-increasing at such points $x_T$. 

We can employ a similar method to show that $f_n'(x_T) \leq 0$ for points $x_T \not \in \bigcup_{k=0}^n T^k(I_\infty)$. This case is simpler than the case we just dealt with in that the $I_\infty$ interval does not split under the first $n$ iterations of $T$. Thus, $$f_n'(x_T) = \sum_{k=1}^{n+1} \left(\frac{d}{dx_T} (T^k(0)) - \frac{d}{dx_T}(T^k(1))\right).$$ 
Again, we have that $$\frac{d}{dx_T} (T^k(0))  = -\rho_A \cdot \rho_1^0 \rho_2^0 \cdots \rho_{k-1}^0  \text{ and }
\frac{d}{dx_T} (T^k(1))  = -\rho_B \cdot \rho_1^1 \rho_2^1 \cdots \rho_{k-1}^1,$$ but now $\rho_i^0 = \rho_i^1$ for all $i$, so $\frac{d}{dx_T} (T^k(0)) - \frac{d}{dx_T}(T^k(1)) = (\rho_B - \rho_A)\overline{\rho}_k < 0$ for all $k$. Thus, $f_n'(x_T) < 0$ and $f_n$ is continuous and decreasing at these $x_T$. 

Finally, it remains to show that $f_n$ is continuous at all $x_T$ that are endpoints of some interval $T^k(I_\infty)$ for $0 \leq k \leq n$. $f_n$ can be seen to be both left-continuous and right-continuous at such points $x_T$ by a simple check that at each of these points, there is a continuous transition between the cases of $I_\infty$ splitting and not splitting as above.
\end{proof}

With this proposition, we can now show that the forward limit set of the orbit of any non-critical point of any $(\rho_A, \rho_B)$-map $T$ has measure zero.

\begin{proposition}
\label{prop:measurezero}
For every non-surjective $(\rho_A, \rho_B)$-map $T$ and any non-critical point $x \in D$, the forward limit set $\omega(x)$ has measure zero. 
\end{proposition}

\begin{proof}
Whenever $x_T \in J \setminus C$, as long as $x_T < \frac{1 - \rho_B}{\rho_A - \rho_B}$, 
we find that $f_n(x_T) \xrightarrow{n \to \infty} 1$. This is a simple calculation due to the result of Theorem \ref{thm:1} which tells us that these maps have an attracting periodic or critical orbit.



We first note that when $\rho_A \geq 1$, $f_n\left(\frac{1 - \rho_B}{\rho_A - \rho_B}\right) = 0$ for all $n$ since $|I_\infty| = 0$. This is the surjective case.  Now we consider any other value of $x_T$. The set $J \setminus C$ is dense, so for any $x_T \in C$ such that $x_T < \frac{1 - \rho_B}{\rho_A - \rho_B}$, 
we can pick $x_T' \in J \setminus C$ with $x_T < x_T'$ so that by Proposition \ref{prop:nonincreasing}, $f_n(x_T) \geq f_n(x_T') \xrightarrow{n \to \infty} 1$.  
Then $f_n(x_T)\xrightarrow{n \to \infty} 1$ for every $x_T < \frac{1 - \rho_B}{\rho_A - \rho_B}$.
This shows that $\bigcup_{k}T^k(I_\infty)$ is full measure for any $x_T \in J$, 
$x_T < \frac{1 - \rho_B}{\rho_A - \rho_B}$. Hence, the limit set $\omega(x)$ has measure zero for every $x \in D$ under any $(\rho_A, \rho_B)$-map $T$ which is not surjective. 
\end{proof}

We would also like to show that the accumulation set of the straight-line flow in $C_2$ directions is a lamination whose cross section is a Cantor set. This is a consequence of the following proposition. 

\begin{proposition}
\label{prop:cantor}
Suppose that $T : D \rightarrow D$ is a non-surjective $(\rho_A, \rho_B)$-map such that the Rauzy induction algorithm goes through infinitely many left and right induction steps. Then, the limit set $\omega(x)$ of the orbit of any non-critical point $x \in D$ is a Cantor set. 
\end{proposition}
\begin{proof}

Recall that, for any $x \in D$, $\omega(x)$ is disjoint from $\bigcup_{k}T^k(I_\infty)$.  
The complement of $\bigcup_{k}T^k(I_\infty)$ is closed and measure zero, so it is nowhere dense. 
Then it is equal to its boundary, which is $\overline{\bigcup_{k}T^k(\partial I_\infty)}$. 
Then we have $\omega(x) \subseteq \overline{\bigcup_{k}T^k(\partial I_\infty)}$.  
We need to establish the reverse inclusion.

Since we do infinitely many left and right inductions, we find that the dilation factors $\rho_A^{m_A}\rho_B^{m_B}$ and $\rho_A^{n_A}\rho_B^{n_B}$ both converge to zero (even if $\rho_A \geq 1$ recall that eventually both intervals contract, unless $x_T$ is all the way to the right in which case $T$ is surjective).
Hence the intervals $I_n$ on which we take the first return map of $T$ with each successive induction are descending intervals whose intersection is the closure of $I_\infty$.  

For some point $x \in D$, for each $k$, let $n_k$ be the first integer such that $T^{n_k}(x) \in I_k$.  
Then we find subsequences of images of $x$ accumulating to the boundary points of $I_\infty$, which are $T(1)$ and $T(0)$.  
Hence the closure of the images of these points $\overline{\bigcup_k T^k(\partial I_\infty)}$ sits in $\omega(x)$.

But this shows that $\omega(x) = \overline{\bigcup_k T^k(\partial I_\infty)}$ for all $x$.  
This is a perfect set because it is equal to $\omega(0)$ and $\omega(1)$, which are themselves the set of limit points of $\bigcup_k T^k(\partial I_\infty)$.
A perfect nowhere dense set in $D$ is necessarily totally disconnected and must be a Cantor set.  So $\omega(x)$ is a measure zero Cantor set which is independent of $x$.
\end{proof}

Finally, we are in position to prove Theorem \ref{thm:2} by combining the results of previous propositions.

\begin{proof}[Proof of Theorem \ref{thm:2}]

We let $X$ be a one-holed dilation torus and $C$ be the Cantor set of directions defined in Theorem \ref{thm:1}. Let $m_B$ be the boundary-parallel direction. In Section \ref{sec:boundaryParallelDirection}, we saw that the dynamics in this direction could be minimal, completely periodic, or accumulate on a periodic orbit.

We then split the rest of $C$ into the set $C_1$ corresponding to $(\rho_A, \rho_B)$-maps that have finitely many left induction steps or finitely many right induction steps, and $C_2$ the remaining directions. We note here that the set $C_1$ is countable. 

Proposition \ref{prop:C1} shows that in the $C_1$ directions, the corresponding $(\rho_A, \rho_B)$-maps accumulate onto a critical orbit. This implies that the straight line flow on $X$ in these directions accumulates onto a saddle connection. 

Propositions \ref{prop:measurezero} and \ref{prop:cantor} together show that in the $C_2$ directions, the corresponding $(\rho_A, \rho_B)$ maps accumulate onto measure zero Cantor sets. This implies that the straight line flow on $X$ in these directions accumulates onto a lamination whose cross section is a measure zero Cantor set. 
\end{proof}

\section{Concluding remarks}

In this paper, we investigated families of non-surjective $2$-AIETs that we called $\mathcal{I}(\rho_A, \rho_B)$. We showed that there is a Cantor set of parameters in each $\mathcal{I}(\rho_A, \rho_B)$ off of which the maps had an attracting periodic orbit. Then we showed that for maps in that Cantor set of parameters but off of a countable subset, orbits are attracted to some measure zero Cantor set. We saw that these results about $2$-AIETs implied similar results for one-holed dilation tori, a family of dilation surfaces.  

There are a few natural questions and further directions of exploration that come from this investigation: 
\begin{enumerate}
    \item Can we expand on the methods of Section \ref{sec:dynamics} to certain families of non-surjective AIETs on more than two intervals generically have attracting periodic orbits? It seems like the key hurdle here would be to prove an inequality like $|H(w)|/|I(w)| > \delta$ as in Proposition \ref{prop:12}. 
    \item It is possible to expand on the methods of Section \ref{sec:cantorMaps} to show that certain non-surjective AIETs on more than two intervals converge to measure zero Cantor sets? 
\end{enumerate}

If we could push the methods used in this paper to understand the limit sets of other families of AIETs, that could open the door to understanding the dynamics on other related families of dilation surfaces as well.  Already, one could use the results in this paper to better understand the dynamics on any dilation surface with a one-holed dilation torus as a subsurface.  By broadening the analysis to answer these questions, we may be able to classify dynamical behaviors on families of dilation surfaces containing slightly more complicated subsurfaces which might exhibit, say, first return maps which are non-surjective 3-AIETs.

\bibliographystyle{plain}
\bibliography{writeup_bib}

\end{document}